\documentclass[reqno, xcolor=pdftex,letterpaper]{amsart}

\usepackage[english]{babel}
\usepackage{amsmath,amsthm,amssymb,amsfonts}
\usepackage{mathrsfs}
\usepackage[utf8x]{inputenc}
\usepackage[normalem]{ulem}
\usepackage{cancel}
\usepackage{csquotes}

\usepackage[usenames,dvipsnames]{color}
\usepackage{enumitem}
\setlist[enumerate]{leftmargin=*}

\usepackage{hyperref}
\usepackage[utf8x]{inputenc}

\usepackage{csquotes}

\usepackage{tikz-cd}

\usepackage{stix}

\theoremstyle{theorem}
\newtheorem{theorem}{\sc \textbf{Theorem}}[section]  
\newtheorem{proposition}[theorem]{\sc \textbf{Proposition}}   
        
\newtheorem{lemma}[theorem]{\sc \textbf{Lemma}}

\theoremstyle{remark}
\newtheorem{definition}[theorem]{\sc \textbf{Definition}}

\newtheorem{remark}[theorem]{\sc \textbf{Remark}}

\usepackage[T1]{fontenc}
\DeclareFontFamily{T1}{calligra}{}
\DeclareFontShape{T1}{calligra}{m}{n}{<->s*[1.44]callig15}{}
\DeclareMathAlphabet\mathcalligra   {T1}{calligra} {m} {n}
\DeclareMathAlphabet\mathzapf       {T1}{pzc} {mb} {it}
\DeclareMathAlphabet\mathchorus     {T1}{qzc} {m} {n}
\DeclareMathAlphabet\mathrsfso      {U}{rsfso}{m}{n}

\makeatletter
\newcommand{\myitem}[1]{%
	\item[#1]\protected@edef\@currentlabel{#1}%
}
\makeatother

\begin{document}
	\title[Surjectivity of convolution operators on harmonic $NA$ groups]{Surjectivity of convolution operators on harmonic $NA$ groups}

	\author[E.\ Papageorgiou]{Effie Papageorgiou}
	\address{Institut f{\"u}r Mathematik, Universit\"at Paderborn, Warburger Str. 100, D-33098
		Paderborn, Germany}
	\email{papageoeffie@gmail.com}

	\thanks{{\em Math Subject Classification} 43A85, 43A90, 22E30}
	\thanks{The author is funded by the Deutsche Forschungsgemeinschaft (DFG, German Research Foundation)--SFB-Gesch{\"a}ftszeichen --Projektnummer SFB-TRR 358/1 2023 --491392403.  This work was also partially supported by Fulbright Greece.} 
	
	\begin{abstract}
		Let $\mu$ be a radial compactly supported distribution on a harmonic $NA$ group. We prove that the right convolution operator $c_{\mu}:f \mapsto f* \mu$ maps the space of smooth $\mathfrak{v}$-radial functions onto itself if and only if the spherical Fourier transform $\widetilde{\mu}(\lambda)$, $\lambda \in \mathbb{C}$, is slowly decreasing. As an application, we prove that certain averages over spheres are surjective on the space of smooth $\mathfrak{v}$-radial functions.
	\end{abstract}
	
	\maketitle
	
	\section{Introduction and statement of the results}\label{Sec1}
	
	Let $\mu$ be a fixed distribution in $\mathcal{E}'(\mathbb{R}^n)$. Under what conditions on $\mu$ is the convolution
	operator
	$$c_{\mu} : f \rightarrow f * \mu$$
	surjective as a map from $\mathcal{E}(\mathbb{R}^n)$ to $\mathcal{E}(\mathbb{R}^n)$, or from $\mathcal{D}'(\mathbb{R}^n)$ to $\mathcal{D}'(\mathbb{R}^n)$? In general, $c_{\mu}$ is not injective on $\mathcal{E}(\mathbb{R}^n)$ (take $\zeta$ to be a zero of the Fourier transform $\mathcal{F}\mu$ of $\mu$ and observe that if $f=e^{i\langle \zeta, \cdot\rangle}$, then $c_{\mu}(f)=0$); it can be nevertheless surjective, as for instance when $c_{\mu}$ is  a constant coefficient differential operator, as proved by Ehrepnpreis \cite[Theorem 10]{Ehr54}. 
	This problem was also studied by Malgrange \cite{Mal55} for the aforementioned spaces of functions and distributions,  when $\mu = D \delta_0$, or more generally when $\mu=\sum_{j} D_j\delta_{x_j}$ where $D$ and the $D_j$ are constant coefficient differential operators, and  $\{x_j\}$ is a finite set of points in $\mathbb{R}^n$. (Here $\delta_{x}$ is the Delta distribution at $x$.)
	
	In order to present some conditions which are equivalent to the surjectivity of $c_{\mu}$, we need first a definition:
	
	\begin{definition}
		We say that a function $u:\mathbb{C}^{n}\rightarrow \mathbb{C}$ is slowly decreasing  if there are positive constants $A$, $B$, $C$, and $D$ such that
		$$\sup\{|u(\zeta)|: \, \zeta\in \mathbb{C}^n, \, \|\zeta-\xi\|\leq A\,\log(2+\|\xi)\|\}\geq B\, (C+\|\xi\|)^{-D}$$
		for all $\xi \in \mathbb{R}^n.$
	\end{definition}
	
	Then, we have the following characterization on the surjectivity of  convolution operators. 
	\begin{theorem}(\cite{Ehr60}, Section 2.) Let $\mu \in \mathcal{E}'(\mathbb{R}^n)$. Then the following are equivalent:
		\begin{itemize}
			\item [(a)] The Fourier transform $\mathcal{F}\mu$ is slowly decreasing.
			\item [(b)] The convolution operator $c_{\mu}$ maps $\mathcal{E}(\mathbb{R}^n)$ onto $\mathcal{E}(\mathbb{R}^n)$.
			\item [(c)] The convolution operator $c_{\mu}$ maps $\mathcal{D}'(\mathbb{R}^n)$ onto $\mathcal{D}'(\mathbb{R}^n)$.
			\item [(d)] There is a distribution $S\in \mathcal{D}'(\mathbb{R}^n)$  such that $S * \mu = \delta_0$.
			\item [(e)] The linear map $f* \mu\mapsto f$ is continuous from $c_{\mu}(\mathcal{D}(\mathbb{R}^n))$ to $\mathcal{D}(\mathbb{R}^n)$.
		\end{itemize}
	\end{theorem} 
	
	Recently, this question has been a subject of investigation beyond the Euclidean setting, and more precisely on hyperbolic spaces, or more generally on arbitrary rank noncompact symmetric spaces $G/K$, see \cite{CGK2017} and \cite{GWK2021}. Owing to the underlying algebraic structure of these Riemannian manifolds, one can still define convolution operators, as well as a Fourier--type transform (the so-called spherical transform, or more generally, the Helgason--Fourier transform). Then, the surjectivity of $c_{\mu}=\ast \mu$ on the space of smooth functions on $G/K$, where $\mu\in \mathcal{E}'(G/K)$ is bi-$K$-invariant, can be completely characterized in terms of the spherical transform of $\mu$. Applications include the surjectivity of certain mean value operators \cite{CGK2017}, the proof of existence of fundamental solutions to invariant differential operators on $G/K$, first proved by Helgason in 1964 \cite{Hel64}, or existence and uniqueness of a wave having three given snapshots at three different times on Euclidean space, noncompact symmetric spaces or the sphere \cite{GKCW2023}.
	
	A harmonic $NA$ group is a solvable Lie group with a canonical left Riemannian structure. It is well-known that as Riemannian manifolds, they contain noncompact Riemannian symmetric spaces of rank one \cite{ACD97}, but the latter form only a small subclass of these manifolds. Indeed, most of them are not symmetric, thus providing
	numerous counterexamples to the Lichnerowicz conjecture \cite{DR1}. It is natural, therefore, to consider the problem of surjectivity for convolution operators on harmonic $NA$ groups, as well as interesting, since the concept of a radial function/distribution is not connected with any group action, and there is no group $K$ acting transitively
	on the distance spheres in $NA$. A challenge occurring from this fact is that the explicit form of the so-called \enquote{angular Laplacian} is not yet available, as emphasized in \cite{Camp2014, Camp2016}. As a result, neither the full Paley--Wiener result for functions, which is essential to this work, is available; the state of the art on Paley--Wiener type results concerns $\mathfrak{v}$--radial functions, see Section \ref{Sec 2} for details.  Our main result is the following. 
	
	\begin{theorem}
		Let $\mu$ be a radial compactly supported distribution on a harmonic $NA$ group. Then the right convolution operator $c_{\mu}:f \mapsto f* \mu$ maps the space of smooth $\mathfrak{v}$-radial functions $\mathcal{E}^{ \mathfrak{v}}(NA)$  onto $\mathcal{E}^{\mathfrak{v}}(NA)$ if and only if the spherical Fourier transform $\widetilde{\mu}(\lambda)$, $\lambda \in \mathbb{C}$, is slowly decreasing.
	\end{theorem}

Let us comment on the result above. The direction \enquote{surjectivity implies slow decrease} holds even for the whole space $\mathcal{E}(NA)$, without restricting to the subspace of $\mathfrak{v}$-radial functions, see Section \ref{Sec 5}. For the other direction, this technical restriction is a consequence, as already mentioned, of the current state of the art on the Paley--Wiener theorem for compactly supported functions, and thus, of our result for compactly supported distributions in Section \ref{Sec 3}.
	
	This paper is organized as follows. After the present introduction in Section \ref{Sec1} and preliminaries in Section \ref{Sec 2}, we prove in Section \ref{Sec 3} that the slow decrease of the spherical Fourier transform $\widetilde{\mu}(\lambda)$ is sufficient to ensure that the convolution operator $c_{\mu}=\, * \mu $ is surjective on the space of smooth $\mathfrak{v}$-radial functions. To do so, we prove a Paley--Wiener theorem for compactly supported distributions, which is of independent interest. In Section \ref{Sec 4} we prove that the slow decrease of $\widetilde{\mu}$ is also necessary for the surjectivity of $c_{\mu}$ on $\mathcal{E}^{\mathfrak{v}}(NA)$. Finally, in Section \ref{Sec 5}, we obtain as an application the surjectivity of certain averages over spheres. These operators are the analog of mean value operators over translated $K$-orbits of a fixed point on symmetric spaces $G/K$, \cite{CGK2017}.
	
	Throughout this paper, the notation $A\lesssim$ between two positive expressions means that
	there is a constant $C > 0$ such that $A \leq CB$. The notation $A \asymp B$ means that $A \lesssim B$ and $B \lesssim A$.
	Moreover, we use the following standard spaces on a smooth manifold $\mathbb{M}$: the space $\mathcal{E}(\mathbb{M})$ denotes the space of smooth functions on $\mathbb{M}$ equipped with the topology of uniform convergence of all derivatives on every compact subset of $\mathbb{M}$, while the space  $\mathcal{D}(\mathbb{M})$ denotes the subspace of functions in $\mathcal{E}(\mathbb{M})$ which are compactly supported. The dual $\mathcal{D}'(\mathbb{M})$ is called the space of distributions on $\mathbb{M}$, and the dual $\mathcal{E}'(\mathbb{M})$ is the space of compactly supported distributions. Finally, the dual spaces $\mathcal{D}'(\mathbb{M})$ and $\mathcal{E}'(\mathbb{M})$ are (if not stated otherwise) equipped with the weak* topology.
	
	\section{Preliminaries}\label{Sec 2}
	A harmonic $NA$ group, also known as a Damek--Ricci space, is a solvable Lie group, equipped with a left-invariant metric. More precisely, $A \ltimes N$ is a semi-direct product of a simply connected Lie group $ A \simeq \mathbb{R}$ with a Heisenberg type Lie group $N$. For more details, we refer the reader to the lecture notes \cite{Rou}, as well as to the articles \cite{ADY}, \cite{ACD97}, \cite{CDKR} and \cite{DR92}. Generalizing hyperbolic
	spaces, harmonic $NA$ spaces provide a large class of examples of Riemannian
	harmonic manifolds which are not symmetric spaces (for harmonic manifolds, see for instance \cite{PS15}).
	Extending to all harmonic $NA$ spaces classical results about hyperbolic geometry and
	harmonic analysis (geodesics, spherical functions, Radon
	transform) entails new difficulties, because of the lack of the compact group $K$ of
	symmetric spaces $G/K$. For analogies and differences to symmetric spaces, we refer the reader to the important paper \cite{CDKR}.
	
	Let us recall the structure of a Heisenberg type Lie group. Suppose that $\mathfrak{n}$ is a two-step nilpotent Lie algebra, equipped with an inner product $  \langle \;, \; \rangle$. Let us denote by $\mathfrak{z}$  the center of $\mathfrak{n}$ and by $\mathfrak{v}$ the orthogonal complement of $\mathfrak{z}$ in $\mathfrak{v}$ (so that $[\mathfrak{v},\mathfrak{v}] \subset \mathfrak{z}$), with $k= \text{dim} \mathfrak{z}$ and $m=\text{dim}\mathfrak{v}$.

	Let $J_{Z}: \mathfrak{v} \rightarrow \mathfrak{v} $ be the linear map defined by
	\[\
	\langle J_{Z}X, Y \rangle = \langle Z, [X,Y] \rangle ,
	\]
	for $X,Y \in \mathfrak{v}, Z \in \mathfrak{z}$.
	Then $\mathfrak{v}$ is of \textit{Heisenberg type} if the following condition is satisfied:
	\[\
	J_Z^2= -|Z|^2I \quad \text{ for every } Z \in \mathfrak{z}.
	\]
	The corresponding connected Lie group $N$ is then called of \textit{Heisenberg type}, and we shall identify $N$ with the Lie algebra $\mathfrak{n}$ via the exponential map:
	\begin{eqnarray*}
		\mathfrak{v} \times \mathfrak{z}  \rightarrow & N \\
		(X,Z)  \mapsto & \exp(X+Z).
	\end{eqnarray*}
	In other words, we realize an $H$-type group $N$ as $\mathbb{R}^m\times \mathbb{R}^k$, for some $m, k\in \mathbb{N}$, via the exponential map. Under this identification
	the Haar measure on $N$ is the Lebesgue measure $dX  dZ$.
	Also, multiplication in $N \equiv \mathfrak{n} = \mathfrak{v} \oplus\mathfrak{z} $  is given by
	\[\
	(X,Z)(X',Z')=(X+X',Z+Z'+\frac{1}{2}[X,X']).
	\]
	Notice that each $J_Z$, with $Z\in \mathfrak{z}$ a unit vector,  induces a complex structure on $\mathfrak{v}$, since $J_{Z}^2=-I_{\mathfrak{v}}$. Therefore, we deduce that $m=\text{dim}\mathfrak{v}$ is always even.

	Consider the group $A=\mathbb{R}_{+}=\{a_t=e^t: \, t\in \mathbb{R}\}$ and let $\mathfrak{a}$ be its Lie algebra. Let $H$ be a vector in $\mathfrak{a}$, acting on $\mathfrak{v}$ with eigenvalue $1/2$, and on $\mathfrak{z}$ with eigenvalue $1$; we
	extend the inner product on $\mathfrak{n}$ to the algebra $\mathfrak{n} \oplus \mathfrak{a}$, by requiring $\mathfrak{n}$ and $\mathfrak{a}$ to be
	orthogonal and $H$ to be a unit vector. The product in
	$N \ltimes \mathbb{R}_{+}$ is given by
	\begin{equation}\label{eq: product law}
		(X,Z,a)(X',Z',a')=(X+a^{1/2}X',Z+aZ'+\frac{1}{2}a^{1/2}[X,X'],aa'),
	\end{equation}
	(notice that $A$ normalizes $N$, so $AN=NA$). Then, $S:=NA$ is a solvable connected and simply connected Lie group, nonunimodular, with Lie algebra $\mathfrak{s}=\mathfrak{n}\oplus \mathfrak{a}=\mathfrak{v} \oplus \mathfrak{z} \oplus \mathbb{R}$ and Lie bracket
	\[\
	[ (X,Z,l), (X',Z',l')]=( \frac{1}{2} l X' -\frac{1}{2} l' X , lZ'-l'Z+[X,X'] ,0).
	\]
	Let us endow $NA$ with the left invariant Riemannian metric which agrees with the inner product on $\mathfrak{s}$ at the identity and it is induced by
	\[\
	\langle (X,Z,l), (X',Z',l') \rangle= \langle X, X' \rangle + \langle Z, Z' \rangle +ll'
	\]
	on $\mathfrak{s}$. Define $d$ to be the distance induced by this Riemannian
	structure. The Riemannian manifold $(NA, d)$ is then called a \textit{Damek-Ricci space}.
	
	The associated left-invariant Haar measure\index{Haar measure} on $NA$ is given by
	\begin{equation*} 
		a^{-Q}{dX}{dZ}{\frac{da}{a}},
	\end{equation*}
	where $Q=\frac{m}{2}+k$ is the \textit{homogeneous dimension} of $N$, while, writing $a\in A$ as $a=a_t=e^t$, the measure becomes
	$$e^{-Qt}dXdZdt.$$ In fact, the only values that $m$ and $k$ can take can be seen in the following table,
	\begin{center}\begin{tabular}{ |c|c|c|c|c|c|c|c|c| } 
			\hline
			$k$ & $8a+1$ & $8a+2$ &  $8a+3$ & $8a+4$ & $8a+5$ & $8a+6$ & $8a+7$ & $8a+ 8$\\
			\hline 
			$m$   & $2^{4a+1}b$ & $2^{4a+2}b$ & $2^{4a+2}b$ & $2^{4a+3}b$ & $2^{4a+3}b$ & $2^{4a+3}b$ & $2^{4a+3}b$ & $2^{4a+4}b$ \\
			\hline
		\end{tabular}
	\end{center}
	where $a\geq 0$ and $b\geq 1$ are arbitrary integers.
	
	Most Riemannian symmetric spaces $G/K$ of noncompact type and rank one
	fit into this framework. According to the Iwasawa decomposition $G = NAK$,
	they can be realized indeed as  $NA = AN$, with $A \simeq \mathbb{R}$. The group $N$ is abelian for real hyperbolic spaces $\mathbb{H}^{N}(\mathbb{R})$  and of Heisenberg type in the other cases $G/K=\mathbb{H}^{N}(\mathbb{C})$, $\mathbb{H}^{N}(\mathbb{H})$, $\mathbb{H}^{2}(\mathbb{O})$.
	Notice that these classical examples form only a very small subclass of harmonic $NA$ groups, as can be seen by checking the dimensions $(k,m)$, which are equal to $(0, N-1)$, $(1, 2(N-1))$, $(3, 4(N-1))$ and $(7, 8)$ for $\mathbb{H}^{N}(\mathbb{R})$, $\mathbb{H}^{N}(\mathbb{C})$, $\mathbb{H}^{N}(\mathbb{H})$, and $\mathbb{H}^{2}(\mathbb{O})$ respectively \cite[p.645]{ADY}.
	
	\begin{remark}
		Since questions about surjectivity on real hyperbolic spaces $\mathbb{H}^{N}(\mathbb{R})$ have been settled in \cite{CGK2017, GWK2021}, we prefer to disregard these spaces in the present work.
	\end{remark}  
	
	Similarly to the case of hyperbolic spaces, this general group $S=NA$ can be realized as the unit ball in $\mathfrak{s}$, 
	\[\
	B(\mathfrak{s})= \lbrace (X,Z,l) \in \mathfrak{s} :
	|X|^2+|Z|^2+l^2<1  \rbrace,
	\]
	through a \textit{Cayley type transform}:
	\begin{eqnarray*}
		C: \quad S \quad & \rightarrow &\quad B(\mathfrak{s})\\
		\quad x=(X,Z,a) \quad & \mapsto& \quad x'=(X',Z',l'),
	\end{eqnarray*}
	where

	$\begin{cases}
		X'= \lbrace (1+a+\frac{1}{4}|X|^2)^2+|Z|^2 \rbrace ^{-1} \; \lbrace (1+a+\frac{1}{4}|X|^2)- J_z \rbrace X, \\
		Z' = \lbrace (1+a+\frac{1}{4}|X|^2)^2+|Z|^2 \rbrace ^{-1} \; 2Z, \\
		
		l' \; = \lbrace (1+a+\frac{1}{4}|X|^2)^2+|Z|^2 \rbrace ^{-1}  \; \lbrace (-1+(a+\frac{1}{4}|X|^2))^2+|Z|^2 \rbrace .
	\end{cases} $
	
	\bigskip
	In the ball model $B(\mathfrak{s})$, the geodesics passing through the origin are the diameters. The geodesic distance to the origin is given by
	\begin{equation}\label{eq: distance ball}
	r=d(x',0) =\text{log}\frac{1+|x'|}{1-|x'|}, \qquad \rho= |x'|=\sqrt{|X'|^2+|Z'|^2+l'^2}=\text{tanh}\frac{r}{2}.
	\end{equation}
	The Riemannian volume is
	\begin{eqnarray*}
		dV &= &2^n(1-\rho ^2)^{-Q-1}\rho^{n-1}d\rho d \sigma \\
		&= &2^{m+k}(\text{sinh}\frac{r}{2})^{m+k} (\text{cosh}\frac{r}{2})^{k} dr d \sigma,
	\end{eqnarray*}
	where $d \sigma$ denotes the surface measure on the unit sphere $\partial B(\mathfrak{s})$ in $\mathfrak{s}$ and $n=\text{dim}NA=m+k+1$, \cite[p.230]{DR92}, \cite[p.72]{Rou}. 
	
	The volume density in normal coordinates at the origin, and by translation
	at any point, is a purely radial function, which means that $S$ is a harmonic
	manifold. The Lichnerowicz conjecture stated that every simply connected harmonic space is either flat or a rank one symmetric space. By construction, Damek-Ricci spaces are simply connected one-dimensional extensions of Heisenberg-type groups, they have nonpositive sectional curvature (thus, they are non-flat), \cite{D}, but in general they are not symmetric. As a result, they provide counterexemples to the Lichnerowicz conjecture.
	
	A function on $NA$ is called radial if it depends only on the distance to the origin, i.e. if there
	exists a function $f_0$ on $[0,+\infty)$ such that  $f(x)=f_0(d(x,e))$ for all $x\in NA$. For distributions, invariant differential
	operators etc., radiality is defined by means of an averaging operator over spheres. More precisely, let $R : \mathcal{D}(NA) \rightarrow \mathcal{D}(NA)$ be the linear operator defined by
	\begin{equation}\label{eq: radialization op}
		Rf(x)=\int_{\mathbb{S}_{\rho}}f(y)\, d\sigma_{\rho}(y), \quad \rho=d(x,e),
	\end{equation}
	where $d\sigma_{\rho}$ is the surface measure induced by the left-invariant Riemannian metric on
	the geodesic sphere $\mathbb{S}_{\rho} = \{y \in NA : d(y,e) = \rho\}$, normalized by $\int_{\mathbb{S}_{\rho}} d\sigma_{\rho}(y)=1$. Let
	$\mathcal{D}(NA)^{\#}$ denote the subspace of radial functions in $\mathcal{D}(NA)$. Then $R$ is a projection from
	$\mathcal{D}(NA)$ onto $\mathcal{D}(NA)^{\#}$.
	
	The following fundamental properties were established in \cite{DR92}: 
	\begin{itemize}
		\item[(i)] Convolution preserves radiality.
		\item[(ii)] Convolution is commutative on radial objects. In particular, the radial integrable functions on $NA$ form a commutative Banach algebra $L^1(NA)^{\#}$ under
		convolution.
		\item[(iii)] The algebra $\mathbb{D}(NA)^{\#}$ of invariant differential operators on $NA$ which are radial
		(i.e. which commute with the averaging operator) is a polynomial algebra with a single generator, the Laplace-Beltrami operator.
	\end{itemize}
	
	\subsubsection{$\mathfrak{v}$-radial functions.} The aim of this part is to define $\mathfrak{v}$-radial functions, for technical reasons having to do with the fact that the Paley–Wiener Theorem for functions, which is essential to this work, has not been
		proved fully on harmonic $NA$ spaces.  
				
	We say that $f$ is a $\mathfrak{v}$-radial function on $NA$, when $f\left(X_1, Z, a\right)=f\left(X_2, Z, a\right)$ if $|X_1^{\prime}|=|X_2^{\prime}|$. Then $f$ depends only on the variables $\left|X\right|, Z$ and $a$, and we write $f=f\left(\left|X\right|, Z, a\right)$. 

Following \cite[p.226]{DR92}, let us define the projector $\pi$ on $\mathcal{D}(N)$ by 
$$(\pi f)(X, Z)=\int_{S^{m-1}} f(|X|\omega, Z)\, d\sigma(\omega),$$ 
where $m=\textrm{dim}\mathfrak{v}$ and $d\sigma$ is the normalized surface measure on the sphere $S^{m-1}$. Consequently, the $\pi$-radial functions are the functions that depend only on $|X|$ and $Z$. Then $\pi$ is an averaging projector on $\mathcal{D}(N)$ in the sense of \cite[pp. 215-16]{DR92}. 

Let now $f\in \mathcal{D}(NA)$. Following \cite{ADiB99}, for $a, a' \in A$ define $f_{a,a'}$ to be the function on $N$ given by 
$$f_{a,a'}(n):=f(ana') \quad \forall n\in N.$$
 With a slight abuse of notation, let $\pi f_{a}$ be the function on $NA$ defined by 
$$(\pi f_{a})(na'):=(\pi f_{a,a'})(n) \quad \forall na' \in NA.$$
Then for $a\in A$, using the group law on $NA$, one sees that $\pi f_{a}$ is $\mathfrak{v}$-radial on $NA$ if and only if $\pi f_{a,a'}$ is $\pi$-radial on $N$ for all $a'\in A$. Finally, for $f, g\in \mathcal{D}(NA)$, we have
\begin{align*}
	(f\ast g) (na)&=\int_{0}^{\infty}\int_{N} f(n'a')\,g(a'^{-1}n'^{-1}na)\,a^{-Q-1}\,dn'da' \\
	&=\int_{0}^{\infty} f_{e,a'}\ast g_{a'^{-1},a}(n)\,a^{-Q-1}\, da'.
	\end{align*}
If $f, g$ are $\mathfrak{v}$-radial on $NA$, then $f_{e,a'}, g_{a'^{-1},a}$ are $\pi$-radial on $N$ for all $a, a'\in A$; thus, since $\pi$ is an averaging projector on $N$, it follows as in \cite[p.229]{ADiB99} that $\pi(f\ast g)_a=(f\ast g)_a$ for all $a\in A$, which implies that $(f\ast g)_a$ is $\pi$-radial on $N$ for all $a\in A$, thus $f\ast g$ is $\mathfrak{v}$-radial on $NA$.

If we denote by $\mathcal{D}^{\mathfrak{v}}(NA)$ to be the set of $\mathfrak{v}$-radial functions in $\mathcal{D}(NA)$, then we have
$\pi:\mathcal{D}(NA)\rightarrow \mathcal{D}^{\mathfrak{v}}(NA)$; moreover
$\langle \pi\varphi, \psi  \rangle= \langle \varphi,\pi\psi  \rangle$ for all $\varphi, \psi \in \mathcal{D}(NA)$. By the last two relations, we will still denote the extension by $\pi: \mathcal{D}'(NA)\rightarrow \mathcal{D}'(NA)$. Then we say that a distribution $T$ is $\mathfrak{v}$-radial if it is in the image of $\pi$, i.e., if
$\pi T=T$. 

	\subsubsection{Derivatives.} 
Let $\{H, e_1, \ldots, e_m, u_1, \ldots, u_k\}$ be an orthonormal basis of the algebra $\mathfrak{s}$ such that $\{e_1, \ldots, e_m\}$ is an orthonormal basis of $\mathfrak{v}$ and $\{u_1, ..., u_k\}$ is an orthonormal basis of $\mathfrak{z}$. Let $\mathbb{X}_0, \mathbb{X}_1, \ldots, \mathbb{X}_m, \mathbb{X}_{m+1}, \ldots \mathbb{X}_{n-1}$ (recall that $m+k=n-1$) be the left invariant vector fields on $S$ which agree with $H, e_1, \ldots, e_m, u_1, \ldots, u_k$ at the identity. By straightforward computations, it follows that if $f\in\mathcal{C}_{c}^{\infty}(NA)$, then
	\begin{equation}\label{diffoper}
		\mathbb{X}_0=a\,\partial_{a}, \quad \mathbb{X}_{\ell}=\sqrt{a}\left( \partial_{X_{\ell}}+\frac{1}{2}\sum_{i=1}^{m}\langle J_{u_i}X, e_{\ell}\rangle\partial_{Z_i} \right), \quad \mathbb{X}_{m+i}=a\,\partial_{Z_i},
	\end{equation}
	for $\ell=1, ..., m$ and $i=1, ..., k$, see \cite[Lemme 17]{Rou} or  \cite[Section 1.1]{BS2022}.
	For $\nu\in \mathbb{N}$ and a multi-index $J \in \{0, . . . , n-1\}^{\nu}$
	we shall write $(J_j)$ for $J$’s $j$-th component and $|J|$ for its length, i.e., $|J|=\nu$. By $\mathbb{X}_J$ we shall mean the left-invariant differential operator
	\[
	\mathbb{X}_J=\mathbb{X}_{(J_1)}\mathbb{X}_{(J_2)}...\mathbb{X}_{(J_\nu)}.
	\]

	\begin{remark}\label{rmk: diff oper}
		The left invariant differential operator $\mathbb{Y}$ defined by $Y\in \mathfrak{n}\oplus\mathfrak{a}$ is given by 
		$$\mathbb{Y}f(x)=\frac{d}{dt}\Bigr|_{t=0}f(x\cdot \exp(tY)),$$
		see for instance \cite[p.82]{Rou}. Therefore, if $$Y=\left(\sum_{\ell=1}^{m}\varepsilon_{\ell} e_{\ell}, \, \sum_{i=1}^{k}\zeta_j 	
		u_j, \, e^{\eta}\right), \quad e_{\ell}, \, \zeta_j,  \, \eta\in \mathbb{R},$$
		then using the group law \eqref{eq: product law} in $NA$ one can compute that 
		$$\mathbb{Y}=\sum_{\ell=1}^{m}\varepsilon_{\ell}\mathbb{X}_{\ell}  + \sum_{i=1}^{k}\zeta_i\mathbb{X}_{m+i} + \eta\mathbb{X}_0.$$	
	\end{remark}

Let $r = r(X, Z, a) = d((X, Z, a), e)$ denote the Riemannian distance of $(X,Z,a)$ to the origin. Using \eqref{eq: distance ball} on the ball model, the Cayley transform and the fact that $J_Z^2 X=-|Z|^2X$, $|J_Z X|=|Z|\, |X|$ and that $J_Z X$ is orthogonal to $X$, one can get (see \cite[p.72]{Rou})  
	\begin{align}\label{distance r/2}
		\cosh^2\left(\frac{r}{2}\right)&=\left( \frac{a^{1/2}+a^{-1/2}}{2}+\frac{1}{8}a^{-1/2}|X|^2\right)^2+\frac{1}{4}a^{-1}|Z|^2 \notag \\
		&= \frac{1}{4}a^{-1}\left\{  \left( 1+a+\frac{1}{4}|X|^2\right)^2 + |Z|^2 \right\}.
	\end{align}
	Then, as proved in \cite[Lemma 3.4]{BS2022}, it holds
	\begin{equation}\label{deriv-coshr}
		|\mathbb{X}_J(\cosh r)|\lesssim \cosh r, \quad \forall J\in \{0,1,..., n-1\}^{\nu}, \, \forall \nu\in \mathbb{N}.
	\end{equation}
	The following result will be needed later.
	
	\begin{lemma}\label{lemma: ball est}
		Let $na\in NA$ such that $d(na, e)\leq R$. Then, 
		$$e^{-R}\leq a\leq e^{R},$$
		while, writing $n=(X,Z)$, we have
		$$|X|\lesssim e^{R/2}, \quad |Z|\lesssim e^{R}.$$ 
	\end{lemma}
	\begin{proof}
		Since it holds $d(a,e)\leq d(na,e)$ for all $n\in N$ and $a\in A$, having $d(na,e)\leq R$ implies by \eqref{distance r/2} that 
		$$\left(\frac{a^{1/2}+a^{-1/2}}{2}\right)^2\leq \cosh^2(R/2).$$
		Since the left hand side rewrites as $\cosh^2 (\log \sqrt{a})$, we get immediately that $e^{-R}\leq a\leq e^R$.
		
		Let now $n=(X,Z)$. Then, due to \eqref{distance r/2}, we have 
		$$\frac{1}{8}a^{-1/2}|X|^2\leq \cosh\frac{R}{2} \quad \text{and} \quad \frac{1}{4}a^{-1}|Z|^2\leq \cosh^2\frac{R}{2},$$
		which, using the fact that $a\leq e^R$, implies
		$$|X|\lesssim e^{R/2} \text{ and } |Z|\lesssim e^{R}.$$
	\end{proof}

	\subsection{Fourier analysis} 
	In this section we review Fourier analysis on harmonic $NA$ groups, and we present Fourier-type transformations for general functions on $NA$, but mainly discuss the so-called spherical analysis, which concerns radial functions on $NA$. 
	As already mentioned, unlike the case of Riemannian symmetric spaces, the concept of a radial function is not connected with any group action and the elementary spherical functions are not necessarily matrix entries of irreducible unitary representations.

	Let us start by defining the \textit{Poisson kernel} $\mathcal{P}$ as follows, \cite[p.233]{CDKR}: if $x=na$, then 
	\begin{align*}
		\mathcal{P}(x,n_0)=\mathcal{P}(na,n_0):=P_a(n_0^{-1}n),
	\end{align*}
	where,  if $n=(X,Z)$, we define
	\begin{align}
		P_a(n)&=P_a(X,Z):=c_{m,k}\, a^{Q}\left( \left( a+\frac{|X|^2}{4} \right)^2+|Z|^2 \right)^{-Q} \label{Poissondef} \\
		&= a^{-Q}P_1(a^{-1}na).\notag
	\end{align} 
	Here, $c_{m,k}=2^{k-1}\pi^{-n/2}\Gamma(n/2)$, \cite[Eq. (2.5)]{ACD97}.	Notice that 
	\begin{equation}\label{eq: integral 1}
		\int_N P_{a}(n)\,dn=1 \quad \text{and} \quad P_a(n^{-1})=P_a(n) \quad \forall a\in A,
	\end{equation}
	see \cite[p.410]{ACD97} for more details.

	\begin{definition}(\cite[p.410]{ACD97}) \label{def: HF transf}
		The Helgason-Fourier transform of a function $f$ in $\mathcal{D}(NA)$ is given by 
		\begin{equation}\label{FTgen}
			\widetilde{f}(\lambda, n_0)=\int_{NA} f(x)\mathcal{P}_{\lambda}(x,n_0)\,dx, \quad \lambda\in \mathbb{C},\; n_0\in N,
		\end{equation}
		where 
		\[
		\mathcal{P}_{\lambda}(x, n_0)=\left[ \mathcal{P}(x,n_0)\right]^{1/2-i\lambda/Q}.
		\]
	\end{definition} 
	
	An interesting property of the function $x \mapsto \mathcal{P}_{\lambda}(x,n_0)$ is that it is constant on certain hypersurfaces which are analogs of hyperplanes in $\mathbb{R}^n$ and of horospheres in Riemannian symmetric spaces of noncompact type. We need to introduce the notion of the geodesic inversion to illustrate this feature. The geodesic inversion $\sigma : NA \rightarrow NA$ is an involutive, measure-preserving diffeomorphism which is explicitly given by 
	\begin{align*}
		\sigma(X,Z,a) &= \left(\left( e^{t}+\frac{|X|^2}{4}   \right)^2+|Z|^2\right)^{-1} \left(\left(-\left(a+\frac{|X|^2}{4}\right)+J_Z    \right) X, -Z, a  \right),
	\end{align*}
	see \cite{Rou}.	For $x = na_t \in NA$, let $A(x) := \log a_t = t$. Then  (see \cite[p.233]{CDKR} but also \cite[p.4273]{RS09})
	$$\mathcal{P}_{\lambda}(x,n_0)=e^{(\frac{Q}{2}-i\lambda)\,A(\sigma(n_0^{-1}x))}.$$
	It is now clear that $\mathcal{P}_{\lambda}(\cdot, n_0)$ is constant on the hypersurfaces $H_{n_0, a_t}=\{n_0\sigma(a_t n): \, n\in N\}$.
	
	A spherical function $\Phi$ on $NA$ is a radial eigenfunction of the Laplace-Beltrami operator normalized so that $\Phi(e)=1$. For $\lambda$ in $\mathbb{C}$, we denote by $ \varphi_{\lambda}$ the spherical
	function with eigenvalue $-(\lambda^2+Q^2/4)$. For all $\lambda \in \mathbb{C}$, by \cite[Proposition 4.2]{ACD97}, it has the following integral representation 
	\begin{equation}\label{spherical intN}
		\varphi_{\lambda}(x^{-1}y)=\int_N \mathcal{P}_{\lambda}(x,n)\mathcal{P}_{-\lambda}(y,n) \, dn, \quad \text{for all} \quad x, y \in NA.
	\end{equation}
	
	For $\alpha, \beta  > -1/2$, $\lambda \in \mathbb{C}$ and $t \geq 0$, let $\phi_{\lambda}^{(\alpha, \beta)}(t)$ denote the Jacobi function
	defined by 
	$$\phi_{\lambda}^{(\alpha, \beta)}(t):={}_2F_1 \left( \frac{1}{2}(\alpha+\beta+1-i\lambda),  \frac{1}{2}(\alpha+\beta+1+i\lambda); \alpha+1; -\sinh^2 t \right)$$
	where ${}_2F_1 (a, b; c; z)$ is the Gaussian hypergeometric function. (For a detailed account on Jacobi functions we refer to \cite{Koo}.) It follows from properties of ${}_2F_1$ that
	$\phi_{\lambda}^{(\alpha, \beta)}(t)=\phi_{-\lambda}^{(\alpha, \beta)}(t)$ for $\alpha, \beta$ and $\lambda$ as above. 
	We recall that for specific parameters $\alpha, \beta$, Jacobi functions coincide with spherical
	functions on $NA$, realized as functions on nonnegative real numbers through the polar decomposition of $NA$. More precisely, in 
	our parametrization, they are related in the following way (\cite[p.650]{ACD97}):
	$$
	\varphi_{\lambda}(t)=\phi_{2\lambda}^{(\alpha, \beta)}(t/2) \quad  \text{where} \quad \alpha = (m+k-1)/2, \quad \beta = (k-1)/2.
	$$	
	\begin{definition}(\cite[p.408]{ACD97}.)
		The spherical transform of a radial function $f$ in $\mathcal{D}(NA)$ is given by
		\begin{equation}\label{FTsph}
			\widetilde{f}(\lambda):=\int_{NA} f(x) \varphi_{\lambda}(x)\,dx, \quad \lambda\in \mathbb{C}.
		\end{equation}
	\end{definition} 
	
	We next define the Abel transform, which maps radial functions $f\in \mathcal{D}(NA)$ to even functions $\mathcal{A}f\in \mathcal{D}(\mathbb{R})$ according to the rule
	\begin{equation*}
		\mathcal{A}f(t):=e^{-\frac{Q}{2}t}\int_N f(ne^t)\, dn,
	\end{equation*}
	\cite[p.650]{ADY}. The following formula describes a \enquote{projection-slice} relation between the Fourier and the spherical transform:
	\begin{equation}\label{eq: Abel interw}
		\widetilde{f}(\lambda)=(\mathcal{F}\circ \mathcal{A})(\lambda)=\int_{-\infty}^{+\infty}e^{-i\lambda t}	\mathcal{A}f(t) \, dt, \quad f\in \mathcal{D}(NA) \quad \text{radial},
	\end{equation}
	which will be essential to this work. Notice that the dual $\mathcal{A}^{*}$ of the Abel transform is given by 
	$$\mathcal{A}^{*}f(x)=\mathcal{A}^{*}f(d(x,e))=\mathcal{A}^{*}f(r)=\int_{\mathbb{S}_r} e^{\frac{Q}{2}A(y)}\,f(A(y))\, d\sigma_{r}(y).$$
	The dual Abel transform is a topological isomorphism between the space $\mathcal{C}^{\infty}(\mathbb{R})_{\text{even}}$ of smooth even functions on the real line and the space $\mathcal{C}^{\infty}(NA)^{\#}$ of smooth radial functions on $NA$, \cite[Theorem 2.1]{AnShift}, due to the duality relation 
	$$\int_{-\infty}^{+\infty}\mathcal{A}f(r)g(r)\, dr=\int_{NA}f(x)\mathcal{A}^{*}g(x)\, dx. $$
	
	In order to compare the Helgason-Fourier transform \eqref{FTgen} with the spherical 
	transform \eqref{FTsph} for radial functions on $NA$, let us recall \cite[Propostition 3.1]{ACD97}, which states that for $f\in \mathcal{D}(NA)$ radial, we have
	\begin{equation}\label{radialtr}
		\widetilde{f}(\lambda, n_0)=\mathcal{P}_{\lambda}(e,n_0)\widetilde{f}(\lambda).
	\end{equation}
	
	Finally, it is well-known that the convolution in $\mathcal{D}(NA)$ is not commutative (unlike in its subset containing radial functions).
	However, if the second factor is radial, we have the following result.
	\begin{proposition}(\cite[Proposition 3.2]{ACD97}) If $f$, $g$ are in $\mathcal{D}(NA)$ and $g$ is radial, then
		\[\widetilde{(f* g)}(\lambda, n_0)=\widetilde{f}(\lambda, n_0)\widetilde{g}(\lambda), \quad \forall \lambda \in \mathbb{C}, \; \forall n_0 \in N.\]
	\end{proposition}

	We now recall the inversion formula for the
	Fourier transform.
	\begin{theorem}(\cite[Theorem 4.4]{ACD97}) Every function $f$ in $\mathcal{D}(NA)$ can be written as
		\begin{equation}\label{inversion}
			f(x)=\frac{c_{m,k}}{4\pi}\int_{-\infty}^{+\infty}\int_N \mathcal{P}_{-\lambda}(x,n)\widetilde{f}(\lambda, n)|c(\lambda)|^{-2}\,d\lambda dn, \quad x\in NA, 
		\end{equation}
		where $|c(\lambda)|^{-2}$ is the Plancherel density.
	\end{theorem}
	
	The Helgason-Fourier transform extends naturally to compactly supported distributions on $NA$: if $T\in \mathcal{E}'(NA)$, then $\widetilde{T}$ is the function on $\mathbb{C}\times N$ given by
	\[
	\widetilde{T}(\lambda, n_0):=T(\mathcal{P}_{\lambda}(\cdot, n_0)).
	\]
	This definition is often written in integral notation as
	\begin{equation}\label{FTdist}
		\widetilde{T}(\lambda, n_0)=\int_{NA}\mathcal{P}_{\lambda}(x, n_0)\, dT(x).
	\end{equation}
	
	In the case that the distribution $T$ is radial, we can consider its spherical Fourier transform, given by
	\begin{equation}\label{distr spherical}
		\widetilde{T}(\lambda):=T( \varphi_{\lambda})=\int_{NA} \varphi_{\lambda}(x)\,dT(x).
	\end{equation}
	
	Finally, for $T\in \mathcal{E}'(NA)$ and $f\in \mathcal{D}(NA)$ radial, using arguments similar to those of \cite[Proposition 3.2]{ACD97}, we have
	\begin{align}\label{conv-mult}
		\widetilde{(T* f)}(\lambda, n_0)=\widetilde{T}(\lambda, n_0)\widetilde{f}(\lambda).
	\end{align}
	Here, we used that if $T\in \mathcal{E}'(NA)$ and $f\in \mathcal{E}(NA)$, then $T* f$ is the function in $\mathcal{E}(NA)$ well-defined by the expression
	\begin{equation*}
		(T* f)(x)=\int_{NA}f(y^{-1}x)\,dT(y).
	\end{equation*}
	
	Finally, given a function $F\in \mathcal{E}(\mathbb{R})$, define the function $\overline{F}$ on $NA$ by $\overline{F}(x):=F(A(x))$, $x\in NA$. For a radial distribution $\mu\in \mathcal{E}'(NA)$, define the distribution $\overline{\mu}\in \mathcal{E}'(\mathbb{R})$ by $\overline{\mu}(F)=\mu(\overline{F})$, for $F\in \mathcal{E}(\mathbb{R})$. Then, we define the Abel transform $\mathcal{A}\mu$ of $\mu$ by
	$$\mathcal{A}\mu=e^{\frac{Q}{2}\cdot }\,\overline{\mu}.$$ 
	The projection-slice theorem for radial distributions $\mu\in \mathcal{E}'(NA)$ now writes
	$$\mathcal{F}(\mathcal{A}\mu)(\lambda)=\widetilde{\mu}(\lambda), \quad \lambda \in \mathbb{C}.$$

	\subsection{Radon transform}	
	
	Let $f\in \mathcal{D}(NA)$ and define its Radon transform by
	\begin{equation}\label{Radon f}
		\widehat{f}(a, n_0):=e^{-\frac{Q}{2}\log a}\int_{N}f(n_0\sigma(na))\, dn.
	\end{equation} 
	(see \cite[Eq. (31)]{Rou} or \cite[Eq. (5.1)]{RS09}). Therefore, the map $f\mapsto \widehat{f}$ is a continuous linear map from $\mathcal{D}(NA)$ into $\mathcal{D}(A\times N)$. Its adjoint (of the first kind) is given by
	\begin{align*}
		\widecheck{f}(x)&=\int_{N}g(n, A(\sigma(n^{-1}x)))\,\mathcal{P}_{0}(x,n)\,dn=\int_{N}g(n, A(\sigma(n^{-1}x)))\,e^{\frac{Q}{2}A(\sigma(n^{-1}x))}\, dn.
	\end{align*}
	
	In fact, the following duality relation holds
	\cite[Proposition 5.2]{RS09}: if $f$ is a measurable function on $NA$ and $g$ is a measurable function on $N \times \mathbb{R}$, then
	$$\int_{N \times \mathbb{R}}\widehat{f}(e^{\log s},n)g(n,s)\, ds=\int_{NA}f(x)\widecheck{g}(x)\, dx.$$
	Finally, for $f\in\mathcal{D}(NA)$, it holds 
	\begin{equation}\label{weq: Radon interw}
		\mathcal{F}(\widehat{f}(\cdot, n))(\lambda)=\widetilde{f}(\lambda, n)	
	\end{equation}
	for $\lambda\in \mathbb{R}$, $n\in N$, \cite[Theorem 5.3]{RS09}.

	Passing to distributions, in order to derive a projection-slice theorem, we first define ``restriction'' maps $T\mapsto\widehat{T}_{n_0}$ from $\mathcal{E}'(NA)$ to $\mathcal{E}'(\mathbb{R})$ for each $n_0\in N$ as follows: fix $n_0\in N$. Then for $F\in \mathcal{E}(\mathbb{R})$, consider the function $F^{n_0}\in \mathcal{E}(NA)$ given by
	\begin{equation}\label{CGK 22}
		F^{n_0}(x):=F(A(\sigma(n_0^{-1}x)))\,e^{\frac{Q}{2}A(\sigma(n_0^{-1}x))}.
	\end{equation}
	Furthermore, for $n_0\in N$ fixed, the function $F^{n_0}$ is constant on the hyperplanes $H_{n_0, a_t}=\{n_0\sigma(a_tn): \, n\in N\}$, and the map $F\mapsto F^{n_0}$ is linear and continuous from $\mathcal{E}(\mathbb{R})$ to  $\mathcal{E}(NA)$.  According to \cite[Proposition 5.2]{RS09}, the following duality relation holds true:
	\begin{equation}\label{eq: dual Radon 2nd}
		\int_{\mathbb{R}} \widehat{f}(a_t,n_0)F(t)\, dt=\int_{NA} f(x)F^{n_0}(x)\, dx.
	\end{equation}

	Let $T\in \mathcal{E}'(NA)$. We define the map $T\mapsto \widehat{T}_{n_0}$ to be the adjoint of the map $F\mapsto F^{n_0}$:
	\begin{equation}\label{CGK 23}
		\widehat{T}_{n_0}(F)=T(F^{n_0}) \quad (F\in \mathcal{E}(\mathbb{R})).
	\end{equation}
	The map $T\mapsto\widehat{T}_{n_0}$ is therefore a continuous linear map from $\mathcal{E}'(NA)$ to $\mathcal{E}'(\mathbb{R})$.
	The projection-slice theorem for distributions thus reads as follows,
	\begin{equation}\label{CGK 24}
		\widetilde{T}(\lambda, n_0)=\mathcal{F}(\widehat{T}_{n_0})(\lambda ).
	\end{equation}
	Last, notice that if $T, \mu\in \mathcal{E}'(NA)$ and $\mu$ is in addition radial, then
	\begin{align*}
		\mathcal{F}((\widehat{T* \mu})_{n_0})(\lambda)=\widetilde{T* \mu}(\lambda,n_0)=\widetilde{T}(\lambda,n_0)\widetilde{\mu}(\lambda)=\mathcal{F}(\widehat{T}_{n_0})(\lambda)\,\mathcal{F}(\mathcal{A}\mu)(\lambda),
	\end{align*}
	where $\mathcal{A}\mu$ is the Abel transform of the radial measure $\mu$, which in turn implies that 
	\begin{equation}\label{CGK 25}
		(\widehat{T* \mu})_{n_0}= \widehat{T}_{n_0} * {\mathcal{A}\mu}.
	\end{equation}

	\section{Paley-Wiener type results} \label{Sec 3}
	In this section we discuss some theorems of Paley-Wiener type, concerning functions and distributions on suitable classes.
	
	\subsection{Paley-Wiener type results for functions.} We first provide some Paley–Wiener type results for functions describing the range of the spherical transform (\ref{FTsph}) and the Fourier transform (\ref{FTgen}).
	
	Let $\mathcal{D}(NA)^{\#}$ denote the space of radial, compactly supported $C^{\infty}$ functions on $NA$. Then, we have the following result for the spherical transform, \cite[Theorem 2.1]{Koo} (see also \cite[Th\'eor\`eme 14]{Rou}).
	\begin{theorem}(\cite{Koo}, \cite{Rou}) 
		The spherical transform $u \mapsto \widetilde{u}$ is a bijection of $\mathcal{D}(NA)^{\#}$ onto the space of entire, even functions $f$ of $\lambda \in \mathbb{C}$, for which there exists a constant $A_f\geq 0$ such that for all integers $j\geq 0$,
		\begin{equation}\label{exptypeRadial}
			\sup_{\lambda\in \mathbb{C}\, }e^{-A_{f}|	\textrm{Im}\lambda|}(1+|\lambda|)^j|f(\lambda)|<\infty.
		\end{equation}
		In addition, the support of $u$ is contained in the closed ball $\overline{B}_R=\{x\in NA: \; d(x,e)\leq R\}$ if and only if  $\widetilde{u}$ verifies (\ref{exptypeRadial}) with $A_{\widetilde{u}}\leq R$.
	\end{theorem}

	We now pass to the general case, which concerns compactly supported functions which are not necessarily radial. To this end, let us  introduce some terminology.
	
	\begin{definition}
		For fixed $R > 0$, define  $\mathcal{H}_R(\mathbb{C} \times N)$ to be the space of $C^{\infty}$ functions $\psi$ on $\mathbb{C} \times N$, holomorphic in the first variable and such that for each integer $j\geq 0$,
		\begin{equation}\label{exptype}
			\sup_{(\lambda, n_0)\in \mathbb{C} \times N} P_{1}(n_0)^{-1/2-(\textrm{Im}\lambda/Q)}\, e^{-R|\textrm{Im}\lambda|}\,(1+|\lambda|)^j|\,\psi(\lambda, n_0)|<\infty,
		\end{equation}
		where $P_1(n_0)$ is the Poisson kernel defined in \eqref{Poissondef}.
		Moreover, define $\mathcal{H}(\mathbb{C} \times N)$ to be the space of all functions $\psi$ in $\cup_{R>0}\mathcal{H}_R(\mathbb{C} \times N)$ satisfying the condition
		\begin{equation}\label{even}
			\int_N \mathcal{P}_{-\lambda}(x,n)\psi(\lambda,n)\,dn=  \int_N \mathcal{P}_{\lambda}(x,n)\psi(-\lambda,n)\,dn
		\end{equation}
		for all $x\in NA$.
	\end{definition} 

	Then, the corresponding Paley--Wiener theorem for compactly supported \textit{$\mathfrak{v}$-radial} functions, related to the Helgason-Fourier transform (\ref{FTgen}) is the following.
	\begin{theorem}(\cite{ACD97}, \cite{Camp2014}) \label{Paley-Wiener}
		The Fourier transform $f(x)\mapsto \widetilde{f}(\lambda, n_0)$ is a bijection of $\mathcal{D}^{\mathfrak{v}}(NA)$ onto the set of holomorphic functions $\mathfrak{v}$-radial in $n_0$, satisfying the conditions (\ref{exptype}) and (\ref{even}). Moreover, $\widetilde{f}$ satisfies (\ref{exptype}) if and only if $f$ has support in the closed ball $\overline{B}_R=\{x\in NA: \; d(x,e)\leq R\}$.
	\end{theorem}
	The fact that $\text{supp}f\subset \overline{B}_R$ implies that $\widetilde{f}$ is holomorphic satisfying the size estimate (\ref{exptype}) was proved in \cite[Theorem 4.5]{ACD97}, without any restrictions on the symmetry of $f$.For the other direction, we refer to the works of Camporesi \cite{Camp2014}, \cite{Camp2016}: first,  in \cite{Camp2014} it is proved that if the function
		$(\lambda, n) \mapsto \psi(\lambda, n)$ satisfies a uniform exponential type $R$ bound in $\lambda$, together with a suitable
		parity condition, and $\psi$ is biradial in $n$, then $\psi$ is the Helgason transform of a regular function
		with support in the geodesic ball of radius $R$. This result is refined in \cite{Camp2016}, relaxing
		the biradiality condition in $n$ to a $\mathfrak{v}$-radiality condition.

	\subsection{Paley-Wiener type theorems for distributions}
	
	Our aim is to give a Paley-Wiener result for distributions, related to the Helgason-Fourier transform (\ref{FTdist}). To this end, let us introduce some necessary definitions.

	\begin{definition}
		A $C^{\infty}$ function $\psi$ on $\mathbb{C} \times N$, holomorphic in the first variable, is called a \textit{holomorphic function of uniform exponential type and of slow growth} if there exist constants $R\geq 0$, $M\in \mathbb{Z}_{+}$, $C>0$ such that
		\begin{equation}\label{exptype2}
			|\psi(\lambda, n_0)|\leq C 
			e^{R|\textrm{Im}\lambda|}\,(1+|\lambda|)^M \,P_1(n_0)^{1/2+(\textrm{Im}\lambda/Q)}, \; \lambda\in \mathbb{C}, \; n_0\in N.
		\end{equation}
		Given $R\geq 0$, let $\mathcal{K}_R(\mathbb{C} \times N)$ denote the space of those  functions $\psi$ satisfying (\ref{exptype2}) and let $\mathcal{K}(\mathbb{C} \times N)$ be their union over all $R\geq 0$. Finally, we denote by $\mathcal{K}(\mathbb{C} \times N)_{\text{even}}$ the space of functions in $\mathcal{K}(\mathbb{C} \times N)$ satisfying (\ref{even}).
	\end{definition}
	
	The next lemma, giving estimates of the derivatives of $\mathcal{P}_{\lambda}(\cdot, n_0)$, $n_0\in N$, is essential. 
	
	\begin{lemma}\label{claim} Fix $R>0$ and let $B_{R}=\{na:\, d(na,e)<R\}$. Then, for every multi-index $J$, there is a constant $c=c(R,J)>0$ and an integer $M=M(J)$ such that for all $na\in B_R$ and $n_0\in N$,
		$$|\mathbb{X}_J\mathcal{P}_{\lambda}(na,n_0)|\leq c\, e^{R|\textrm{Im}\lambda|}\, (1+|\lambda|)^M\,P_{1}(n_0)^{1/2+(\textrm{Im}\lambda)/Q}.$$
	\end{lemma}
	
	\begin{proof}
		Recall that by Definition \ref{def: HF transf}, $$\mathcal{P}_{\lambda}(na,n_0)=\left[P_a(n_0^{-1}n)^{-1/Q}\right]^{-Q/2+i\lambda}.$$
		For every multi-index $J \in \{0, . . . , n-1\}^{\nu}$, one can straightforwardly compute that the derivatives $\mathbb{X}_J\mathcal{P}_{\lambda}(\cdot, n_0)$, $n_0\in N$, have the following expansion,
		\begin{align}\label{product}
			\mathbb{X}_J\mathcal{P}_{\lambda}(na,n_0)=\sum_{s=1}^{\nu}P_s(\lambda) \, \left[P_a(n_0^{-1}n)^{-1/Q}\right]^{-Q/2+i\lambda-s} \, D_{\nu+1-s}(P_a(n_0^{-1}n)^{-1/Q}).
		\end{align}
		
		Here $P_s(\lambda)$ is a polynomial in $\lambda$ of degree $s$, while for a smooth function $f$ on $NA$, we denote $$D_{1}f:=\prod_{i=1}^{\nu}\mathbb{X}_{(i)}f,$$ 
		where $\mathbb{X}_{(i)}$ denote the components of $\mathbb{X}_{J}$,  $$D_{\nu}f:=\mathbb{X}_{J}f,$$
		and $D_{\nu+1-s}f$, $s=2, ..., \nu-1$, is a linear combination of products of the form 
		$$\mathbb{X}_{I_1}f...\mathbb{X}_{I_s}f,\quad \mathbb{X}_{I_r}\in \{0,, ..., n-1\}^{|I_r|}, \quad r=1,...,s,$$
		with $|I_1|+...+|I_s|=|J|=\nu$, $|I_1|, ..., |I_s|\geq 1$, but at least one of $|I_1|,..., |I_s|$ is strictly larger than $1$ (in other words, at least one of those derivatives is of order higher than one).

		Therefore, in view of \eqref{product}, it suffices to show that for $d(na, e)<R$, and any $s\in \{1, ... \nu\}$, it holds
		$$|P_s(\lambda)|\lesssim_{s,R}(1+|\lambda|)^{s},$$
		\begin{equation}\label{eq: Plambda bounds}
			\left|\left[P_a(n_0^{-1}n)^{-1/Q}\right]^{-Q/2+i\lambda}\right|=|\mathcal{P}_{\lambda}(na, n_0)|\lesssim_{s,R} e^{R|\textrm{Im}\lambda|}P_1(n_0)^{1/2+(\textrm{Im}\lambda)/Q},
		\end{equation}
		and
		\begin{equation}\label{eq: Plambda s}
			\left|\left[P_a(n_0^{-1}n)^{-1/Q}\right]^{-s}D_{\nu+1-s}(P_a(n_0^{-1}n)^{-1/Q})\right|\lesssim_{s,R} 1,
		\end{equation} 
		whence the Lemma follows. 
		We stress that the implied constants can depend on $s$, $R$ and the geometry, but not on $n_0$.

		The claim for the upper bound of the polynomials $P_s(\lambda)$ is obvious. Thus, let us now prove \eqref{eq: Plambda bounds}. 
		Recall first by \cite[Corollary 3.2]{AD08} that 
		$$e^{-Qd(na,e)}\leq \frac{P_a(n_0^{-1}n)}{P_1(n_0^{-1})}=\frac{\mathcal{P}(na, n_0)}{\mathcal{P}(e,n_0)}\leq e^{Qd(na,e)}.$$
		Therefore, taking into account that $d(na,e)\leq R$, we obtain indeed
		\begin{align}
			|\mathcal{P}_{\lambda}(na,n_0)| &=\left|	\frac{\mathcal{P}_{\lambda}(na,n_0)}{\mathcal{P}_{\lambda}(e, n_0)} \right| |\mathcal{P}_{\lambda}(e, n_0)| \notag \\
			&=\left| \left(  \frac{P_a(n_0^{-1}n)}{P_1(n_0^{-1})}  \right)^{\frac{1}{2}-i\frac{\lambda}{Q}}   \right||\mathcal{P}_{\lambda}(e, n_0)| \notag \\
			&\leq e^{QR/2} \, e^{R \, |\textrm{Im} \lambda|} \, P_{1}(n_0)^{1/2+(\textrm{Im}\lambda)/Q}. \label{P Iml}
		\end{align}

		It remains to prove \eqref{eq: Plambda s}. To do so, let us start with some pointwise estimates of $\mathcal{P}(\cdot, n_0)^{-1/Q}$. We claim that if $d(na, e)\leq R$ and $r=d(n_0^{-1}na, e)$, then 
		\begin{equation}\label{coshr vs Poisson}
			e^{-2R} \cosh r \lesssim P_a(n_0^{-1}n)^{-1/Q} \lesssim \cosh r.
		\end{equation}
		Indeed, to show the right-hand side inequality, it suffices to use the distance formula \eqref{distance r/2} with the Poisson kernel formula \eqref{Poissondef} to immediately get 
		\begin{equation*}
			P_a(n_0^{-1}n)^{-1/Q}\lesssim \cosh r.
		\end{equation*}
		For the opposite direction, observe first that by Lemma \ref{lemma: ball est} we have
		$$\frac{1+a}{a}\leq e^{R}+1\leq 2e^{R}.$$ Then, since for $\alpha, \beta>0$ having
		$\frac{\alpha}{\beta}\leq M $ for some  $M\geq 1$, implies  $\frac{\alpha+\gamma}{\beta+\gamma}\leq M$ for all $\gamma \geq 0,$
		we get 
		$$\frac{\left(1+a+\frac{1}{4}|X|^2\right)^2}{\left(a+\frac{1}{4}|X|^2\right)^2}\leq 4e^{2R}, $$
		which in turn yields the left-hand side of \eqref{coshr vs Poisson}.

		It remains therefore to deal with the derivatives of $\mathcal{P}(\cdot, n_0)^{-1/Q}$.  Observe first that if $x=(X,Z,a)$ and $d=d(x,e)$, then \eqref{distance r/2} implies that 
		\begin{equation*}
			2\cosh d = a^{-1}\left\{ \left(a+\frac{|X|^2}{4} \right)^2+|Z|^2 \right\} +a^{-1}\left( 1+\frac{|X|^2}{2} \right).
		\end{equation*}
		Therefore, by (\ref{Poissondef}), for $n=(X,Z)$ and  $n_0=(X_0, Z_0)$, we have that
		\[
		c_{m,k}^{1/Q}\,P_a(n_0^{-1}n)^{-1/Q}=-a^{-1}\left( 1+\frac{|X-X_0|^2}{2}\right)+2\cosh r, \quad r=d(n_0^{-1}na,e).
		\]
		It follows that for all $\mu\in \mathbb{N}$, using the notation of \eqref{diffoper}, we have
		\begin{align*}
			c_{m,k}^{1/Q}\,\mathbb{X}_0^\mu(P_a(n_0^{-1}n)^{-1/Q})&=(-1)^{\mu-1}a^{-1}\left( 1+\frac{|X-X_0|^2}{2}\right)+2\mathbb{X}_0^{\mu}\cosh r.
		\end{align*}
		For derivatives coming from $\mathfrak{z}$, we compute
		\begin{align*}
			c_{m,k}^{1/Q}\,\mathbb{X}_{m+i}(P_a(n_0^{-1}n)^{-1/Q})&=2\mathbb{X}_{m+i}\cosh r, \quad i=1, ..., k, 
		\end{align*}
		while for derivatives coming from $\mathfrak{v}$, we have
		\begin{align*}
			c_{m,k}^{1/Q}\,\mathbb{X}_{\ell}(P_a(n_0^{-1}n)^{-1/Q})&=- a^{-1/2}(X_{\ell}-X_{0\ell })+2\mathbb{X}_{\ell}\cosh r,\\	
			c_{m,k}^{1/Q}\,\mathbb{X}_{\ell'}\mathbb{X}_{\ell}(P_a(n_0^{-1}n)^{-1/Q})&=-\delta_{\ell' \ell}+2\mathbb{X}_{\ell'}\mathbb{X}_{\ell}\cosh r, \quad \ell, \ell'=1,..., m,
		\end{align*}
		where $\delta_{\ell' \ell}=1$, if $\ell' = \ell$, and $0$ otherwise (here, $X_{\ell}$ and $X_{0\ell}$ denote the $\ell$ coordinates of $X$ and $X_0$, respectively, in the basis $\{e_{1}, ..., e_{m}\}$). Also, observe that these derivatives are enough for our purposes, as any higher order derivatives of $\mathcal{P}(\cdot, n_0)^{-1/Q}$ will consist of linear combinations of the terms $a^{-1}\left( 1+\frac{|X-X_0|^2}{2} \right)$ and $a^{-1/2}(X_{\ell}-X_{0\ell})$, as well as of  derivatives of $\cosh r$, so no additional terms to be controlled will appear.
		
		We are now ready to prove \eqref{eq: Plambda s}. First, observe first that all derivatives of $\cosh r$ can be bounded above by a multiple of $\cosh r$, due to \eqref{deriv-coshr}.
		Next, by the distance formula \eqref{distance r/2}, since $r=d(n_0^{-1}na,e)$, we obtain
		\[
		\frac{1}{4}a^{-1}\left(1+\frac{|X-X_0|^2}{4}\right)^2\leq \cosh^2\left(\frac{r}{2}\right).
		\]
		This also yields on the one hand, that 
		$$a^{-1}\left(1+\frac{|X-X_0|^2}{2}\right)\leq a^{-1}\left(1+\frac{|X-X_0|^2}{4}\right)^2\leq 4\cosh^2\left(\frac{r}{2}\right),$$
		and on the other hand, that  
		$$a^{-1/2}|X_{\ell}-X_{0\ell}|\leq a^{-1/2}|X-X_0|\leq a^{-1/2}\left(1+\frac{|X-X_0|^2}{4}\right) \leq 2\cosh\left(\frac{r}{2}\right).$$
		Altogether, the above computations give  the following upper bounds,
		\begin{equation}\label{deriv-PoissonQ}
			|\mathbb{X}_{J'}(P_a(n_0^{-1}n)^{-1/Q})|\lesssim_{J'} \cosh r, \qquad r=d(n_0^{-1}na,e).
		\end{equation}
		for any $J'\in \{0, ..., n-1\}^{|J'|}$. In other words, the order of the derivative of $P_a(n_0^{-1}n)^{-1/Q}$ does not quantitatively affect the size of the upper bound in \eqref{deriv-PoissonQ}. Therefore
		\begin{equation}\label{eq: Di}
			D_{\nu+1-s}(P_a(n_0^{-1}n)^{-1/Q})\lesssim (\cosh r)^{s}.
		\end{equation}
		By \eqref{coshr vs Poisson} and \eqref{eq: Di} we conclude that for any $s\in \{1, ..., \nu\}$,
		\begin{align}\label{together s}
			\left|	\left[P_a(n_0^{-1}n)^{-1/Q}\right]^{-s}D_{\nu+1-s}( P_a(n_0^{-1}n)^{-1/Q})\right|\lesssim e^{2sR}\lesssim_{s, R} 1,
		\end{align}
		that is, we have proved \eqref{eq: Plambda s}. The proof of the Lemma is now complete.
	\end{proof}

	The following proposition yields a Paley--Wiener theorem for $\mathfrak{v}$-radial compactly supported distributions (the space of which will be denoted by $\mathcal{E}_{\mathfrak{v}}'(NA)$) related to the Helgason-Fourier transform (\ref{FTdist}). Its analog on symmetric spaces of noncompact type can be found in \cite[Corollary 5.9]{Hel2000}. (See also \cite{Dad79} and \cite{Eg73} for the Paley--Wiener theorem for distributions on noncompact symmetric spaces.) 
	
	\begin{proposition}\label{PWHelDistr}  The Helgason-Fourier transform $T\mapsto \widetilde{T}$ is a linear bijection of $\mathcal{E}_{\mathfrak{v}}'(NA)$ onto the subspace of $\mathcal{K}(\mathbb{C} \times N)_{\text{even}}$ functions which are $\mathfrak{v}$-radial in the second variable. Moreover, $\widetilde{T}$ satisfies (\ref{exptype2}) if and only if $\text{supp}T\subset \overline{B}_R$.
	\end{proposition}

	\begin{proof}
		We follow the approach of \cite[Corollary 5.9]{Hel2000}. Let $\eta_{\epsilon}\in \mathcal{D}(NA)$ be radial, positive, such that  $\int\eta_{\epsilon}=1$ and  $\text{supp}(\eta_{\epsilon})\subset \overline{B}_{\epsilon}(e)$. Then 
		\[
		\widetilde{\eta_{\epsilon}}(\lambda)=\int_{B_{\epsilon}(e)}\eta_{\epsilon}(x)( \varphi_{\lambda}(x)-1)\,dx+1.
		\]
		Then, $\widetilde{\eta_{\epsilon}}(\lambda)\rightarrow 1$, as $\epsilon\rightarrow 0$, uniformly on compact sets. If $T\in \mathcal{E}'(NA)$, then by $(\ref{conv-mult})$ we have $\widetilde{T*\eta_{\epsilon}}(\lambda, n_0)=\widetilde{T}(\lambda, n_0)\widetilde{\eta_{\epsilon}}(\lambda)$, so for every $n_0\in N$, $\lambda \mapsto \widetilde{T}(\lambda, n_0)$ is holomorphic. Further, $T\mapsto \widetilde{T}$ is injective and $\widetilde{T}$ satisfies (\ref{even}). Last, $n_0\mapsto \widetilde{T}(\lambda, n_0)$ is $\mathfrak{v}$-radial in $n_0$ (see the method of \enquote{descend} to a hyperbolic space, \cite[Lemma 3.1]{Camp2016}).
		
		We first show that $\lambda \mapsto \widetilde{T}(\lambda, n_0)$ satisfies (\ref{exptype2}) when $\text{supp} T\subset \overline{B}_R$. Since $T\in \mathcal{E}'(NA)$, by continuity we can find differential operators $D_i\in \textbf{D}(N)$, $E_i\in \textbf{D}(A)$ such that  
		\begin{equation}\label{Hel42}
			|T(\zeta)|\leq \sum_{\text{finite}}\sup_{x}|\tau(D_i)\tau(E_i)(\zeta)(x)|, \quad \text{for} \quad \zeta\in \mathcal{D}(NA).
		\end{equation}
		Here, we used the notation 
		$$(\tau(X_1...X_{\ell})f)(na)=\left\{ \frac{\partial^{\ell}}{\partial t_1... \partial t_{\ell}}f(na(\exp t_{1}X_{1})...(\exp t_{\ell}X_{\ell})) \right\}_{t_1=\cdots=t_{\ell}=0}$$ for  $X_1,...,X_{\ell}\in \mathfrak{n}\oplus\mathfrak{a}.$ 
		
		Consider a function $\Psi \in \mathcal{E}(\mathbb{R})$  such that $0 \leq \Psi \leq 1$, $\Psi\equiv 1$ on $(-\infty, 1/2)$ and $\Psi\equiv 0$ on $[1,\infty)$. Fix $n_0 \in N$ and define 
		\[
		x \mapsto \Psi_{\lambda}(x)=\mathcal{P}_{\lambda}(x,n_0)\,\Psi(|\lambda|(d(x,e)-R)).
		\]
		Then $\Psi_{\lambda}\in \mathcal{D}(NA)$ and $\Psi_{\lambda}(x)=\mathcal{P}_{\lambda}(x,n_0)$ in a neighborhood $B_{R'}$ of $\overline{B}_R$.  Therefore, $\widetilde{T}(\lambda, n_0)=T(\Psi_{\lambda})$ which we estimate by (\ref{Hel42}) with $\zeta=\Psi_\lambda$. The support of $\Psi_\lambda$ equals $d(x,e)\leq R+|\lambda|^{-1}$. Let $|\lambda|\geq 1$. Notice that each $(X,Z)\in \mathfrak{n}$ defines a left invariant operator related to a linear combination of the derivatives $\mathbb{X}_1, ..., \mathbb{X}_{n-1}$; in the same way, each $\log a\in \mathfrak{a}$ defines a left invariant operator related to a multiple of $\mathbb{X}_0$ (see Remark \ref{rmk: diff oper}).
		Thus, by Lemma \ref{claim}, we get that for all $i\in \mathbb{N}$, there is $M=M(i)\in \mathbb{Z}$ such that $$|\tau(D_i)\tau(E_i)(\mathcal{P}_{\lambda}(\cdot, n_0))(x)|\leq C_{R,i} \, e^{(R+|\lambda|^{-1})|\textrm{Im}\lambda|}\,(1+|\lambda|)^M\,P_{1}(n_0)^{1/2+(\textrm{Im}\lambda)/Q},$$ 
		for all $x\in NA$ such that $d(x,e)\leq R+|\lambda|^{-1}$, while the constant $C_{R,i}$ does not depend on $n_0$. Also, it is clear that derivatives of $\Psi(|\lambda|(d(\cdot, e)-R))$ introduce polynomials in $|\lambda|$. Therefore, since the index $i$ runs in a finite subset of $\mathbb{N}$, we conclude that $\widetilde{T}(\lambda, n_0)$ satisfies (\ref{exptype2}) for all $|\lambda|\geq 1$, thus for all $\lambda\in \mathbb{C}$.

		For the converse direction, let $\psi$ be a function on $\mathbb{C}\times N$ which is $\mathfrak{v}$-radial in the second variable, and satisfies (\ref{exptype2}) and (\ref{even}). Then the assignment $$f\rightarrow T(f)=\int_{-\infty}^{+\infty}\int_N \mathcal{P}_{-\lambda}(e,n)\psi(\lambda, n)\widetilde{f}(-\lambda, n)|c(\lambda)|^{-2}\, dnd\lambda $$ is a well defined functional on $\mathcal{D}(NA).$ To see this, observe on the one hand that for all $j\in \mathbb{N}$ there exists $C_j>0$ such that
		$$ |\psi(\lambda, n)||\widetilde{f}(-\lambda, n)||c(\lambda)|^{-2}\leq C_j (1+|\lambda|)^{-j}P_{1}(n)^{1/2} $$
		for all $\lambda \in \mathbb{R}$ and $n\in N$, where we used the bound (\ref{exptype2}) for $\psi$, the Paley--Wiener result (\ref{exptype}) for $f$, and the fact that the Plancherel density $|c(\lambda)|^{-2}$ has polynomial growth on the real line (see for instance \cite[p.653]{ADY}). On the other hand, we have 
		$$\int_N |\mathcal{P}_{-\lambda}(e,n)|P_{1}(n)^{1/2}\, dn=\int_N P_{1}(n)\, dn=1,$$
		for $\lambda \in \mathbb{R}$,	by \eqref{eq: integral 1}. This shows that $|T(f)|<\infty$ for all $f\in \mathcal{D}(NA).$

		Now, taking into account the upper bounds of $\psi$ given in (\ref{exptype2}), and those of $\eta_{\epsilon}$ given in (\ref{exptypeRadial}) (recall that $\eta_{\epsilon}$ is radial, supported on $\overline{B}_{\epsilon}$), we get that for every $M\in \mathbb{Z}^+$, there is a constant $c_M$ such that
		$$|\psi(\lambda, n)\widetilde{\eta_{\epsilon}}(\lambda)|\leq c_M(1+|\lambda|)^{-M}\,e^{(R+\epsilon)|\textrm{Im}(\lambda)|}\,P_1(n)^{1/2+(\textrm{Im}\lambda/Q)}.$$
		Therefore, since the map $n \mapsto \psi(\lambda, n)\widetilde{\eta}(\lambda)$ is $\mathfrak{v}$-radial,  Theorem \ref{Paley-Wiener} implies that $\psi\widetilde{\eta_{\epsilon}}$ is the Helgason--Fourier transform of a function $f_{\epsilon}\in \mathcal{D}^{\mathfrak{v}}(NA)$, supported in $B_{R+\epsilon}$. The Plancherel formula (see for instance \cite[Theorem 5.1]{ACD97}), however, implies that this test function considered as a distribution must be $T_{\epsilon}$, with $\text{supp}T_{\epsilon}\subset B_{R+\epsilon}$. But since $\lim_{\epsilon\rightarrow 0}\widetilde{\eta_{\epsilon}}(\lambda)=1$ uniformly on compacts, it follows that  $\lim_{\epsilon\rightarrow 0}T_{\epsilon}(f)=T(f)$ for all $f\in \mathcal{D}(NA)$, so $\text{supp}T\subset B_{R}$. Finally, we verify that the Helgason--Fourier transform of $T$ coincides with $\psi$:
		\begin{align*}
			\widetilde{T}(\lambda,n)&=\int_{NA}\mathcal{P}_{\lambda}(x, n)\,dT(x)=\lim_{\epsilon\rightarrow 0}\int_{NA}\mathcal{P}_{\lambda}(x, n) \,dT_{\epsilon}(x)\\
			&=\lim_{\epsilon\rightarrow 0}\int_{NA}\mathcal{P}_{\lambda}(x, n)f_{\epsilon}(x) \,dx=\lim_{\epsilon\rightarrow 0}\psi(\lambda, n)\widetilde{\eta_{\epsilon}}(\lambda)=\psi(\lambda, n).
		\end{align*}
		The proof is now complete.
	\end{proof}

	\begin{remark}\label{Remark Abel diagram distr}
		The Paley--Wiener Theorem for radial distributions implies that the Abel transform is a linear bijection from $\mathcal{E}'(NA)^{\#}$ onto $\mathcal{E}'(\mathbb{R})_{\text{even}}$, and in fact we have a commutative diagram of linear bijections:
		$$
		\begin{tikzcd}
			& \mathcal{E}'(\mathbb{R})_{\text{even}} \arrow[rd, "\mathcal{F}"] & \\
			\mathcal{E}'(NA)^{\#}	 \arrow[ru, "\mathcal{A}"] \arrow[rr, "\widetilde{ }", swap] & & \mathcal{H}(\mathbb{C})_{\text{even}}
		\end{tikzcd}
		$$
		Here $\mathcal{H}(\mathbb{C})_{\text{even}}$ is the vector space of slowly increasing even holomorphic functions on $\mathbb{C}$ of exponential type.
	\end{remark} 
	
	\section{Slow decrease implies surjectivity} \label{Sec 4}
	Our aim in this section is to give a sufficient condition for the surjectiviy of convolution operators on the space of smooth functions. More precisely, suppose that $\mu \in \mathcal{E}'(NA)$ is radial and let $c_{\mu}$ be the convolution operator on $\mathcal{E}(NA)$ given by $c_{\mu}(f)=f* \mu$. Let us denote be $\mathcal{E}^{\mathfrak{v}}(NA)$ and $\mathcal{E}'_{\mathfrak{v}}(NA)$ the $\mathfrak{v}$-radial subspaces of $\mathcal{E}(NA)$ and $\mathcal{E}'(NA)$, respectively.  Then the main result of this section is the following.
	
	\begin{theorem}\label{thm: templatesurj} Let $\mu$ be a radial distribution in $\mathcal{E}'(NA)$ whose spherical Fourier transform $\widetilde{\mu}(\lambda)$ is a slowly decreasing function on $\mathbb{C}$. Then the convolution operator $c_{\mu}=\,\ast \mu: \mathcal{E}^{\mathfrak{v}}(NA)\rightarrow \mathcal{E}^{\mathfrak{v}}(NA)$ is surjective.
	\end{theorem}
	Notice that since $\mu$ is radial, the fact that $\lambda\mapsto \varphi_{\lambda}$ is even  combined with (\ref{distr spherical}), implies that $\widetilde{\mu}$ is even. Also, by  Proposition \ref{PWHelDistr},  $\widetilde{\mu}(\lambda)$ is of exponential type and slow growth in $\mathbb{C}$.
	
	\begin{proof}
		Assume that $\widetilde{\mu}$ is slowly decreasing. 
		Since $\mu$ is radial, the adjoint map to $c_{\mu}: \mathcal{E}(NA)\rightarrow\mathcal{E}(NA)$ is 
		$$c_{\mu}: \mathcal{E}'(NA)\rightarrow\mathcal{E}'(NA), \quad T\mapsto T*\mu.$$ Indeed, the adjoint of $c_{\mu}$ would correspond to a convolution operator $c_{\mu^{*}}$ with the distribution $\mu^{*}$ being the reflection of $\mu$ with respect to the reflection $na\mapsto a^{-1}n^{-1}$. In view of radiality, the two distributions coincide.
		
		Therefore, it will be sufficient to show the latter map is injective and has closed range in the strong (and hence weak*) topology on $\mathcal{E}_{\mathfrak{v}}'(NA)$. Indeed, the theorem will then follow from the fact that a continuous linear map $\Phi:E \rightarrow F$ between two Fr{\'e}chet spaces is surjective if and only if its adjoint $\Phi^{*}:F' \rightarrow E'$ is injective and has a weak* closed range in $E'$. 
		
		Let us first show that the map $c_{\mu}$ is injective.  The set of all $\lambda$ for which $\widetilde{\mu}(\lambda)\neq 0$ is open and dense in $\mathbb{C}$, and therefore $T* \mu =0$ implies that
		$$\widetilde{T}(\lambda, n_0)\widetilde{\mu}(\lambda)=0, \quad (\lambda, n_0)\in \mathbb{C}\times N,$$
		from which we obtain $\widetilde{T}(\lambda, n_0)=0$, hence $T\equiv 0$.
		
		Next, we claim that the range  $c_{\mu}(\mathcal{E}_{\mathfrak{v }}'(NA))$ is equal to
		\begin{equation}\label{eq: range}
			\{ T\in \mathcal{E}_{\mathfrak{v}}'(NA): \; \widetilde{T}(\lambda, n_0)/\widetilde{\mu}(\lambda) \text{ is holomorphic in } \lambda \text{ for each } n_0\in N \}
		\end{equation}
		and further, that the latter set is a closed subset in $\mathcal{E}_{\mathfrak{v}}'(NA)$.  
		
		First of all, the fact that the range $c_{\mu}(\mathcal{E}_{\mathfrak{v}}'(NA))$  is contained in the set \eqref{eq: range} follows by the observation that if $T_1, T_2\in \mathcal{E}'(NA)$ where $T_1$ is $\mathfrak{v}$-radial and $T_2$ is radial, then $T_1* T_2$ is $\mathfrak{v}$-radial and $\widetilde{(T_1* T_2)}(\lambda, n_0)=\widetilde{T_1}(\lambda, n_0)\widetilde{T_2}(\lambda)$.
		Suppose now that $T$ belongs to the set \eqref{eq: range}.  Since $\widetilde{T}(\lambda, n_0)/\widetilde{\mu}(\lambda)$ is holomorphic for each fixed $n_0$, \cite[Proposition A.1]{CGK2017} implies that $\widetilde{T}(\lambda, n_0)/\widetilde{\mu}(\lambda)$ is smooth on $\mathbb{C}\times N$. Also, since for every fixed $n_0$ the function $\lambda\mapsto P_{1}(n_0)^{1/2-i(\lambda/Q)}$ is also holomorphic and nowhere vanishing, we deduce that the function $P_{1}(n_0)^{-1/2+i(\lambda/Q)}\widetilde{T}(\lambda, n_0)/\widetilde{\mu}(\lambda)$ is also smooth on $\mathbb{C}\times N$. In addition, by the forward Paley–Wiener Theorem on $NA$, there exist $C,R>0$ and $M\in \mathbb{N}$ (neither of which depends on $n_0$) such that $\widetilde{T}$ satisfies the growth condition
		$$|\widetilde{T}(\lambda, n_0)|\leq C 
		e^{R|\textrm{Im}\lambda|}\,(1+|\lambda|)^M\, P_1(n_0)^{1/2+(\textrm{Im}\lambda/Q)}, \; \lambda\in \mathbb{C}, $$
		for all $n_0\in N.$ By \cite[Proposition 2.5]{CGK2017}, there exist $C', R'>0$ and $M'\in \mathbb{N}$ such that
		$$\left|\frac{P_{1}(n_0)^{-1/2+(i\lambda/Q)}\,\widetilde{T}(\lambda, n_0)}{\widetilde{\mu}(\lambda)}\right|\leq C'(1+|\lambda|)^{M'}e^{R'|\textrm{Im}\lambda|}, \quad \lambda\in \mathbb{C}$$
		for all $n_0\in N$. Therefore, $\widetilde{T}(\lambda, n_0)/\widetilde{\mu}(\lambda)\in \mathcal{K}(\mathbb{C}\times N)$ and it remains to show that this quotient also belongs to $\mathcal{K}(\mathbb{C}\times N)_{\text{even}}$. In other words, it remains to show that it verifies $\eqref{even}$.
		
		By assumption, $\widetilde{T}\in \mathcal{K}(\mathbb{C}\times N)$, which means that 
		\begin{equation*}
			\int_N \mathcal{P}_{-\lambda}(x,n_0)\widetilde{T}(\lambda,n_0)\,dn_0=  \int_N \mathcal{P}_{\lambda}(x,n_0)\widetilde{T}(-\lambda,n_0)\,dn_0.
		\end{equation*}
		The fact that $\widetilde{\mu}(\lambda)$ is even gives 
		\begin{equation*}
			\int_N \mathcal{P}_{-\lambda}(x,n_0)\frac{\widetilde{T}(\lambda,n_0)}{\widetilde{\mu}(\lambda)}\,dn_0=  \int_N \mathcal{P}_{\lambda}(x,n_0)\frac{\widetilde{T}(-\lambda,n_0)}{\widetilde{\mu}(-\lambda)}\,dn_0
		\end{equation*}
		for all $\lambda\in \mathbb{C}$ for which $\widetilde{\mu}(\lambda)\neq 0$. In fact, this relation extends to all $\mathbb{C}$ by continuity, the density of $\{\lambda\in \mathbb{C}:\, \widetilde{\mu}(\lambda)\neq 0\}$ in $\mathbb{C}$, and the fact that the integrands above are uniformly continuous on compact sets.
		
		Having shown that $\widetilde{T}(\lambda, n_0)/\widetilde{\mu}(\lambda)\in \mathcal{K}(\mathbb{C}\times N)_{\text{even}}$, since $n_0\mapsto \widetilde{T}(\lambda, n_0)/\widetilde{\mu}(\lambda)$ is $\mathfrak{v}$-radial in $n_0\in N$, the Paley–Wiener result of Proposition \ref{PWHelDistr} for distributions implies the existence of a distribution $S\in \mathcal{E}'(NA)$ such that $\widetilde{S}(\lambda, n_0)=\widetilde{T}(\lambda, n_0)/\widetilde{\mu}(\lambda)$. In other words, $S=T*\mu$, proving the range characterization \eqref{eq: range}.
		
		Finally we prove that the set in \eqref{eq: range} is a closed subset of $\mathcal{E}_{\mathfrak{v}}'(NA)$. 
		Owing to the intertwining property of the Abel transform discussed in Remark \ref{Remark Abel diagram distr}, the problem now becomes Euclidean: indeed, let us denote by $\mathcal{A}\mu$ the Abel transform of the radial distribution $\mu\in \mathcal{E}'(NA)$. Then $\mathcal{A}\mu\in \mathcal{E}'(\mathbb{R})$ has a slowly decreasing Euclidean Fourier transform on $\mathbb{C}$. Hence by the results of Ehrenpreis and H{\"o}rmander (for a unified version of the results, see \cite[Theorem 2.3]{CGK2017}), it follows that the convolution operator $v\mapsto v * \mathcal{A}{\mu}$ on $\mathcal{E}'(\mathbb{R})$ has closed range $c_{\mathcal{A}{\mu}}(\mathcal{E}'(\mathbb{R}))$. 
		Then by the characterization \eqref{eq: range}, \eqref{CGK 24}, and \eqref{CGK 25}, we conclude that
		$$	c_{\mu}(\mathcal{E}_{\mathfrak{v}}'(NA)) = \{T \in \mathcal{E}_{\mathfrak{v}}'(NA): \, \widehat{T}_{n_0}\in c_{\mathcal{A}{\mu}}(\mathcal{E}'(\mathbb{R})) \text{ for all } n_0\in N\}. $$
		For each $n_0\in N$, the linear map $T \mapsto \widehat{T}_{n_0}$ from $\mathcal{E}'(NA)$ to $\mathcal{E}'(\mathbb{R})$ is continuous. Therefore, $c_{\mu}(\mathcal{E}_{\mathfrak{v}}'(NA))$ is a closed subspace of $\mathcal{E}_{\mathfrak{v}}'(NA)$, since $\mathcal{E}'(\mathbb{R}) * \mathcal{A}{\mu}$ is closed in $\mathcal{E}'(\mathbb{R})$.  The proof is now complete.
	\end{proof}
	
	\section{Surjectivity implies slow decrease} \label{Sec 5}
	Let $\mathcal{E}'(NA)^{\#}$ denote the space of radial compactly supported  distributions in $NA$.	For $\mu \in \mathcal{E}'(NA)^{\#}$ fixed, let $c_{\mu} :\mathcal{E}(NA) \rightarrow \mathcal{E}(NA)$ be the convolution operator defined as $c_{\mu}(f) = f * \mu$. In the previous section, we proved that if $\widetilde{\mu}$ is slowly decreasing, then $c_{\mu} : \mathcal{E}^{\mathfrak{v}}(NA) \rightarrow \mathcal{E}^{\mathfrak{v}}(NA)$ is surjective. In this section, we aim to prove the converse -full- assertion, which we state below.
	
	\begin{theorem}\label{thm: converse} Suppose that $\mu\in \mathcal{E}'(NA)^{\#}$. If $c_{\mu} : \mathcal{E}(NA) \rightarrow \mathcal{E}(NA)$ is surjective, then $\widetilde{\mu}$ is a slowly decreasing function on $\mathbb{C}$.
	\end{theorem}
	
	Once again, the idea is to transfer the analysis to $\mathbb{R}$ by means of the Abel transform and then use certain Euclidean results proved in \cite{GWK2021}. To this end, let us introduce the following operator. Let $f\in C^{\infty}(\mathbb{R})$ and recall that the dual Abel transform $\mathcal{A}^{*}$ was defined by
	\begin{equation}\label{eq: dual Abel}
		\mathcal{A}^{*}f(x)=\mathcal{A}^{*}f(d(x,e))=\mathcal{A}^{*}f(r)=\int_{d(e,y)=r}\widetilde{f}(y)\, d\sigma_{r}(y)
	\end{equation}
	where $ \widetilde{f}(y)=\widetilde{f}(X,Y,a)=e^{\frac{Q}{2}\log a}f(\log a)$. The map $\mathcal{A}^{*}$ is  a continuous linear map of Fr{\'e}chet spaces, and for any $f\in \mathcal{E}(\mathbb{R})$ it is clear that $\mathcal{A}^{*}f\in \mathcal{E}(NA)^{\#}$.  Furthermore recall that, as already mentioned, $\mathcal{A}^{*}$ is a topological isomorphism between $C^{\infty}(\mathbb{R})_{\text{even}}$ and $C^{\infty}(NA)^{\#}$.

	\begin{proposition}\label{Prop.3.4}  Let $\mu\in \mathcal{E}'(NA)^{\#}$. Then the following diagram commutes: 
		\begin{center}
			\begin{tikzcd}
				\mathcal{E}(\mathbb{R})_{\text{even}} \arrow[r, "c_{\mathcal{A}\mu}"] \arrow[d, "\mathcal{A}^{*}"]
				& \mathcal{E}(\mathbb{R})_{\text{even}} \arrow[d, "\mathcal{A}^{*}"] \\
				\mathcal{E}(NA)^{\#} \arrow[r, "c_{\mu}"]
				& \mathcal{E}(NA)^{\#}
			\end{tikzcd}
		\end{center}	
	\end{proposition}
	
	\begin{proof}
		Since all maps in the diagram are continuous, it suffices to prove that the diagram consisting of the adjoint maps is commutative:
		\begin{center}
			\begin{tikzcd}
				\mathcal{E}'(NA)^{\#}	 \arrow[r, "c_{\mu}"] \arrow[d, "\mathcal{A}"]
				& \mathcal{E}'(NA)^{\#} \arrow[d, "\mathcal{A}"] \\
				\mathcal{E}'(\mathbb{R})_{\text{even}}	 \arrow[r, "c_{\mathcal{A}\mu}"]
				& \mathcal{E}'(\mathbb{R})_{\text{even}}
			\end{tikzcd}
		\end{center}
		First, let us justify why the maps in the diagram above are the claimed adjoints of the maps in the statement of Proposition \ref{Prop.3.4}. To begin with, as already mentioned, the map $\mathcal{A}^{*}$ is the adjoint of $\mathcal{A}$. Next,  the adjoint of $c_{\mu}$ is itself, as we have proved in Proposition \ref{Prop.3.4}. Finally, the adjoint of $c_{\mathcal{A}\mu}$ is given by convolution with $(\mathcal{A}\mu)^{\widecheck{ }}$, the latter being the reflection of the distribution $\mathcal{A}\mu$ with respect to the map $t\mapsto -t$ in $\mathbb{R}$. Taking into account that the Abel transform maps radial distributions on $NA$ to even distributions on $\mathbb{R}$, we see that $(\mathcal{A}\mu)^{\widecheck{ }}$ and $\mathcal{A}\mu$ coincide. Thus, showing the commutativity of the diagram boils down to proving that
		$$\mathcal{A}T* \mathcal{A}\mu=\mathcal{A}(T* \mu)$$
		for all $T\in \mathcal{E}'(NA)^{\#}$. This identity is clearly true, since the Euclidean Fourier transform of both sides is equal to $\widetilde{T}(\lambda)\widetilde{\mu}(\lambda)$, which finishes the proof.
	\end{proof}

	\begin{proof}[Proof of Theorem \ref{thm: converse}]
		Suppose that $c_{\mu}: \mathcal{E}(NA) \rightarrow \mathcal{E}(NA)$ is surjective. Then for any $\psi \in \mathcal{E}(NA)^{\#}$, there is a $\phi \in \mathcal{E}(NA)$ for which $\phi * \mu = \psi$. Observe that this implies that $\phi*\mu = R(\phi*\mu)$, where $R$ is the radialization operator in \eqref{eq: radialization op}, so by a Fubini argument, it is easy to see that one can replace $\phi$ by $R\phi$. Therefore, we may assume that $\phi \in \mathcal{E}(NA)^{\#}$. Thus $c_{\mu}$ maps $\mathcal{E}(NA)^{\#}$ onto $\mathcal{E}(NA)^{\#}$.
		
		As already mentioned in this section,  the linear map $\mathcal{A}^{*}:\mathcal{E}(\mathbb{R})_{\text{even}} \rightarrow \mathcal{E}(NA)^{\#}$ is a linear bijection. Proposition \ref{Prop.3.4} therefore shows that the convolution operator 
		$c_{\mathcal{A}\mu}:\mathcal{E}'(\mathbb{R})_{\text{even}} \rightarrow \mathcal{E}'(\mathbb{R})_{\text{even}}$ is surjective.  The problem is now Euclidean, thus we can invoke \cite[Theorem 4.1]{GWK2021}, which states, in dimension one, that if a distribution $\Lambda\in \mathcal{E}'(\mathbb{R})_{\text{even}}$ is such that the convolution operator $c_{\Lambda}$ maps $ \mathcal{E}(\mathbb{R})_{\text{even}}$ onto $ \mathcal{E}(\mathbb{R})_{\text{even}}$, then its Euclidean Fourier transform $\mathcal{F}\Lambda$ is slowly decreasing. Therefore, taking $\Lambda=\mathcal{A}\mu$, we conclude that the holomorphic function $\mathcal{F}({\mathcal{A}\mu})=\widetilde{\mu}(\lambda)$ is slowly decreasing, which then proves the desired surjectivity of $c_{\mu}$ on $\mathcal{E}(NA)$. The proof of Theorem \ref{thm: converse} is now complete.
	\end{proof}

	\section{Mean value operators}
	
	In this section, our aim is to prove -as an application of the results of the previous sections- the surjectivity of certain mean value operators on $\mathcal{E}(NA)$. On symmetric spaces of the noncompact type $G/K$, the authors in \cite[Section 9]{CGK2017} consider mean value operators over translated $K$-orbits of a fixed point, and present various surjectivity results for those operators according to the structure of the group $G$. Due to the lack of a stabilizer group $K$ on the present setting, we modify the definition of these mean value operators as follows: for suitable functions $f$ on $NA$, consider the average on a sphere of radius $t > 0$ around $x\in NA$ given by
	$$\mathcal{M}_tf(x) := \int_{d(y,e)=t} f(xy)\, d\sigma_t(y) = f * \sigma_t(x),$$
	where $\sigma_t$ is the normalized surface measure on the sphere of radius $t$ with center at the identity $e$. The spherical Fourier transform of the radial measure $\sigma_t$ is
	$$\widetilde{\sigma_t}(\lambda)=\varphi_{\lambda}(t).$$ 
	
	The main result of this section is the following.
	
	\begin{proposition}\label{prop: surj MVO}
		For any $t>0$, the mean value operator $$\mathcal{M}_t:\mathcal{E}^{\mathfrak{v}}(NA)\rightarrow \mathcal{E}^{\mathfrak{v}}(NA)$$ is surjective.
	\end{proposition}
	
	To prove Proposition \ref{prop: surj MVO}, let us make some preparatory comments. First of all, in view of Theorem \ref{thm: templatesurj} and the fact that the spherical transform of the surface measure $\sigma_t$ is $\varphi_{\lambda}(t)$, it suffices to show that $\lambda \mapsto\varphi_{\lambda}(t)$ is slowly decreasing. As already mentioned, spherical functions $\varphi_{\lambda}$ on harmonic $NA$ groups fall under the scope of Jacobi analysis. We use the following formula by Koornwinder,
	\begin{align}\label{eq: Koo formula}
		&\frac{\Gamma\left( \frac{n-1}{2}\right)\Gamma\left( \frac{1}{2}\right)}{2^{\frac{n-1}{2}}\Gamma\left( \frac{n}{2}\right)}\,(\sinh t)^{n-2}\,(\cosh t)^{\frac{k}{2}}\,\varphi_{\lambda}(t)=  \\
		&\int_{0}^{t}\cos(\lambda s) \,(\cosh t-\cosh s)^{\frac{n-3}{2}}\,{}_2F_{1}\left( \frac{k}{2}, 1-\frac{k}{2}; \frac{n-1}{2};\frac{\cosh t-\cosh s}{2\cosh t}\right) ds. \notag
	\end{align}
	see \cite[Eq. (5.57) and (5.60)]{Koo}. Let $\Phi_k(\lambda)$ denote the integral on the right hand side above:
	\begin{align*}
		&\Phi_k(\lambda):=\int_{0}^{t}\cos(\lambda s) \,(\cosh t-\cosh s)^{\frac{n-3}{2}}\,{}_2F_{1}\left( \frac{k}{2}, 1-\frac{k}{2}; \frac{n-1}{2};\frac{\cosh t-\cosh s}{2\cosh t}\right) ds.
	\end{align*}
	The function $\Phi_k$ can be extended to an even holomorphic function, also denoted by $\Phi_k$, and clearly it is enough to prove that this function is slowly decreasing. Since the case $k=0$ corresponds to the real hyperbolic space, where the surjectivity of $\mathcal{M}_t$ has been treated in \cite{CGK2017}, we may assume that $k\geq 1$.

	\begin{proof}[Proof of Proposition \ref{prop: surj MVO}] We split the proof into two cases, namely when $k$ is odd or even.
		
		Assume first that $k$ is odd. Then since $m$ is always even, it follows that $n=m+k+1$ is even. Then the desired slow decrease in $\lambda$ has been established in \cite[pp.3636--3637]{CGK2017}.
		
		Next, assume that $k$ is even. This implies that $n$ is odd, and in fact, by the table in the Introduction, $n\geq 7$. If $k=2$, then observe that the hypergeometric function is equal to 
		$${}_2F_{1}\left( 1, 0; \frac{n-1}{2};\frac{\cosh t-\cosh s}{2\cosh t}\right)={}_2F_{1}\left( 0, 1; \frac{n-1}{2};\frac{\cosh t-\cosh s}{2\cosh t}\right)=1,$$ due to its symmetry in the first two arguments. Writing $n-3=2N\geq 4$, it follows that 
		\begin{align*}
			\Phi_k(\lambda)=
			\int_{0}^{t}\cos(\lambda s) \,(\cosh t-\cosh s)^{N}\, ds.
		\end{align*}
		Then the desired slow decrease in $\lambda$ has been established in \cite[pp.3635-3636]{CGK2017}.
		
		It remains to treat the case when $k\geq 4$ is even, which in turn implies that $n$ is odd: write $n-3=2N\geq 10$. It follows that $1-\frac{k}{2}=-M$, for some $M\in \mathbb{N}$. Then the hypergeometric function
		$$	{}_2F_{1}\left( -M , 1+M; 1+N;\frac{\cosh t-\cosh s}{2\cosh t}\right)
		$$ 
		reduces to a polynomial
		\begin{align*}
			\sum_{\ell=0}^{M}(-1)^{\ell}{M \choose \ell}\frac{(1+M)_{\ell}}{(1+N)_{\ell}}\left(\frac{\cosh t-\cosh s}{2\cosh t}\right)^{\ell}=\sum_{\ell=0}^{M}c_{\ell, M}\left(\frac{\cosh t-\cosh s}{2\cosh t}\right)^{\ell},
		\end{align*}
		with $c_{0,M}=1$. (Here, we used the notation $(\nu)_{\ell}=\nu(\nu+1)...(\nu+\ell-1)$.) Thus we need to show that the function 
		\begin{align*}
			\Phi_{k}(\lambda)
			=\sum_{\ell=0}^{M}2^{-\ell}\,c_{\ell, M}\,(\cosh t)^{-\ell} \int_{0}^{t}\cos(\lambda s)\,(\cosh t-\cosh s)^{N+\ell}  ds
		\end{align*}
		is slowly decreasing. Clearly, this function cannot be identically zero.
		Setting $$\widetilde{I}_{N+\ell}(\lambda):=\int_{0}^{t}\cos(\lambda s)\,(\cosh t-\cosh s)^{N+\ell}  ds,$$ 
		let us
		rewrite
		\begin{equation}\label{eq: Phi sum}
			\Phi_{k}(\lambda)
			=\sum_{\ell=0}^{M}2^{-\ell}\,c_{\ell, M}\,(\cosh t)^{-\ell}\,\widetilde{I}_{N+\ell}(\lambda).
		\end{equation}
		Assuming for a moment that the following asymptotics for $\widetilde{I}_{N+\ell}$ hold, 
		\begin{align}\label{eq: I nu tilde asymp}
			\widetilde{I}_{N+\ell}(\lambda)&=(N+\ell+1)!\,(\sinh t)^{N+\ell}\,  \lambda^{-N-\ell-1}\,\sin\left( \lambda t-\frac{(N+\ell)\pi}{2}\right) +\textrm{O}(\lambda^{-N-\ell-\frac{3}{2}})
		\end{align}
		as $\lambda \rightarrow +\infty$, let us complete the proof: given \eqref{eq: I nu tilde asymp}, it follows that the main term in the sum \eqref{eq: Phi sum} amounts to $\ell=0$. In other words, 
		\begin{equation}
			\Phi_{k}(\lambda)
			=(N+1)!\,(\sinh t)^{N}\, \lambda^{-N-1}\,\sin\left( \lambda t-\frac{N\pi}{2}\right)+r_{N}(\lambda),
		\end{equation}
		where $r_{N}(\lambda)=\textrm{O}(\lambda^{-N-\frac{3}{2}})$.
		
		For $\xi \in \mathbb{R}$, we define the sets $U_{\xi}$ and $V_{\xi}$ by
		$$U_{\xi} := \{\zeta \in \mathbb{C}: \, |\zeta-\xi| <A \log(2+|\xi|)\}, \quad V_{\xi} := U_{\xi} \cap \mathbb{R},$$
		where $A=(N+1)!\,(\sinh t)^{N}.$ 
		If necessary, we take the above $\xi_0$ large enough so that $tV_{\xi}$ is an interval of length greater than $2\pi$
		and  $A \log(2+\xi)\leq \xi/2$ for $\xi>\xi_{0}$. 
		Therefore, on the one hand, there is a $\lambda_0$ in $V_{\xi}$ such that $\sin\left(\lambda_{0}t-\frac{N\pi}{2}\right) =1$, which yields
		$$\sup_{\lambda \in V_{\xi}}\left|A\, \lambda^{-N-1}\, \sin\left(\lambda t-\frac{N\pi}{2}\right) \right|\geq A\, \lambda_{0}^{-N-1}.$$
		On the other hand, for $\lambda\in V_{\xi}$ we have $\xi/2 \leq \xi-A \log(2+\xi) < \lambda <\xi+A \log(2 +\xi) \leq 2\xi $,
		thus
		\begin{align*}
			\sup_{\zeta\in U_{\xi}}|\Phi_k(\zeta)|&\geq \sup_{\lambda\in V_{\xi}}|\Phi_k(\lambda)|  \\
			&\geq \sup_{\lambda\in V_{\xi}}\left|A\, \lambda^{-N-1}\, \sin\left(\lambda t-\frac{N\pi}{2}\right) \right|- \sup_{\lambda\in V_{\xi}}|r_N(\lambda)|\\
			&\geq A\,\lambda_0^{-N-1}-C(\lambda+1)^{-N-\frac{3}{2}}\\
			&\geq A\,(2\xi)^{-N-1}-C(\xi/2+1)^{-N-\frac{3}{2}}.
		\end{align*}
		When $\xi$ is large enough, the first term above dominates, and this completes the proof that $\Phi_{k}$ is slowly decreasing.
	\end{proof}
	
	We finally complete the proof of Proposition \ref{prop: surj MVO}, proving the asymptotics \eqref{eq: I nu tilde asymp} for $$\widetilde{I}_{\nu}(\lambda):=\int_{0}^{t}\cos(\lambda s)\,(\cosh t-\cosh s)^{\nu}  ds, \quad \nu \in \mathbb{N}.$$
	
	\begin{lemma}\label{lemma: CGK lemma 10.4}
		Let $\nu \in \mathbb{N}$. Then, 
		\begin{align*}
			\widetilde{I}_{\nu}(\lambda)&=(\nu+1)!\,(\sinh t)^{\nu}\,  \lambda^{-\nu-1}\,\sin\left( \lambda t-\frac{\nu\pi}{2}\right)+\textrm{O}(\lambda^{-\nu-\frac{3}{2}})
		\end{align*}
		as $\lambda\rightarrow +\infty.$
	\end{lemma}
	\begin{proof}
		Let us first recall the definition of Bessel functions of the first kind: By \cite[§3.715, Formula 20]{GR2015}, we have
		\begin{equation}\label{eq: Bessel def}
			\int_{0}^{\pi/2}\cos(z\cos\theta)\,(\sin \theta)^{2\mu} d\theta=\sqrt{\pi}\, 2^{\mu-1}\, z^{-\mu}\, \Gamma\left(\mu+\frac{1}{2}\right) J_{\mu}(z),
		\end{equation}
		whenever $\text{Re}\mu >-\frac{1}{2}$. In addition, if $\nu\in \mathbb{N}$, then by \cite[§8.461, Formula 1]{GR2015}, we have 
		\begin{align*}
			J_{\nu+\frac{1}{2}}(z)&=\sqrt{\frac{2}{\pi z}}    \sin\left(z-\frac{\nu\pi}{2}\right) \sum_{m=0}^{ \lfloor \frac{\nu}{2}\rfloor}\frac{(-1)^m(\nu+2m)!}{(2m)!(\nu-2m)!}(2z)^{-2m}   \\
			&+\sqrt{\frac{2}{\pi z}}   \cos\left(z-\frac{\nu\pi}{2}\right) \sum_{m=0}^{\lfloor \frac{\nu-1}{2}\rfloor}\frac{(-1)^m(\nu+2m+1)!}{(2m+1)!(\nu-2m-1)!} (2z)^{-2m-1} .
		\end{align*}
		This implies that as $z>0$ grows to infinity, we have
		\begin{align}\label{eq: Bessel half}
			J_{\nu+\frac{1}{2}}(z)&=\sqrt{\frac{2}{\pi z}}  \sin\left(z-\frac{\nu\pi}{2}\right) +   \textrm{O}(z^{-3/2}).		
		\end{align}
		It follows that for fixed $t>0$ and $\nu\in \mathbb{N}$, we have
		\begin{align}\label{eq: I nu asymp}
			I_{\nu}(\lambda):&=\int_{0}^{t}\cos(\lambda s)\,(t^2-s^2)^{\nu} \, ds \notag \\ &=t^{2\nu+1}\int_{0}^{\pi/2}\cos(\lambda t\cos\theta) \, (\sin\theta)^{2\left( \nu+\frac{1}{2}\right)}\, d\theta \notag \\
			&=\sqrt{\pi}\, 2^{\nu-\frac{1}{2}}\, t^{2\nu+1}\, (\lambda t)^{-\nu-\frac{1}{2}}\, \Gamma\left(\nu+1\right) J_{\nu+\frac{1}{2}}(\lambda t)  \notag \\
			&=2^{\nu}\,(\nu+1)!\, t^{\nu}\,\lambda^{-\nu-1}\,\sin\left(\lambda t-\frac{\nu\pi}{2}\right)+\textrm{O}(\lambda^{-\nu-\frac{3}{2}}), \quad \lambda\rightarrow +\infty,
		\end{align}
		using \eqref{eq: Bessel def} and \eqref{eq: Bessel half}.

		We are now ready to compute the asymptotics for $\widetilde{I}_{\nu}$. To this end, we follow the approach of  \cite[Lemma 10.4]{CGK2017}.	Let 
		$$f(z)=\sum_{\ell=0}^{\infty}\frac{z^{\ell}}{(2\ell)!}.$$
		Observe that $\cosh z= f(z^2)$. Therefore, for any $\nu\in \mathbb{N}$, the function 
		$$G_{\nu}(z)=\begin{cases}
			\left(\frac{f(t^2)-f(t^2-z)}{z}\right)^{\nu} \quad &z\neq 0 \\
			(f'(t^2))^{\nu} \quad &z=0
		\end{cases}
		$$
		is well-defined and holomorphic near $[0,t^2]$, thus it admits a Taylor series expansion 
		\begin{equation}\label{eq: CGK 71}
			G_{\nu}(z)=\sum_{i=0}^{\infty}a_{i,\nu}z^{i}=a_{0, \nu}+z\,g_{\nu}(z),
		\end{equation} 
		where the remainder $g_{\nu}(z)$ is holomorphic near the line segment $[0,t^2]\subseteq\mathbb{C}$ and $a_{0, \nu}=(\sinh t/(2t))^{\nu}>0$. Substituting $z=t^2-s^2$ in \eqref{eq: CGK 71}, we have
		$$\left(\frac{\cosh t-\cosh s}{t^2-s^2}\right)^{\nu}=a_{0, \nu}+(t^2-s^2)\,g_{\nu}(t^2-s^2).$$
		Thus 
		\begin{align*}
			(\cosh t-\cosh s)^{\nu}&=\left(\frac{\cosh t-\cosh s}{t^2-s^2}\right)^{\nu}(t^2-s^2)^{\nu}\\
			&=a_{0,\nu}(t^2-s^2)^{\nu}+(t^2-s^2)^{\nu+1}\,g_{\nu}(t^2-s^2).
		\end{align*}
		Therefore, $\widetilde{I}_{\nu}(\lambda)$ can be rewritten as 
		\begin{align*}
			\widetilde{I}_{\nu}(\lambda)&=a_{0, \nu}\int_{0}^{t}\cos(\lambda s)(t^2-s^2)^{\nu}\,ds +\int_{0}^{t}\cos(\lambda s)\,(t^2-s^2)^{\nu+1}\,g_{\nu}(t^2-s^2)\,ds\\
			&=a_{0, \nu}\, I_{\nu}(\lambda) +R_{\nu}(\lambda).
		\end{align*}
		The term $I_{\nu}(\lambda)$ satisfies the asymptotics \eqref{eq: I nu asymp} as $\lambda\rightarrow +\infty$. For the remainder term $R_{\nu}(\lambda)$, observe that, owing to the fact that $g_{\nu}$ is holomorphic near $[0,t^2]$, we may further expand 
		$$g_{\nu}(t^2-s^2)=g_{0,\nu} +(t^2-s^2)\,\widetilde{g}_{\nu}(t^2-s^2),$$
		where $\widetilde{g}_{\nu}$ is again holomorphic near $[0,t^2]$. This way, we obtain 
		\begin{align*}
			R_{\nu}(\lambda)=g_{0,\nu}\, I_{\nu+1}(\lambda)+\int_{0}^{t}\cos(\lambda s)\,(t^2-s^2)^{\nu+2}\,\widetilde{g}_{\nu}(t^2-s^2)\,ds,
		\end{align*}
		which is $\textrm{O}(\lambda^{-\nu-2})$: for the first summand above, this follows by \eqref{eq: I nu asymp}, while for the second by \cite[Lemma 10.2]{CGK2017}. 
		
		Altogether, we conclude that 
		$$\widetilde{I}_{\nu}(\lambda)=(\nu+1)!\,(\sinh t)^{\nu}\,  \lambda^{-\nu-1}\,\sin\left( \lambda t-\frac{\nu\pi}{2}\right)+\textrm{O}(\lambda^{-\nu-\frac{3}{2}}).$$
	\end{proof}
	
	\textbf{Acknowledgments.} We would like to thank the referees for their valuable comments and constructive input.


\begin{thebibliography}{99}
		
		
		\bibitem{ADY} J.-Ph. Anker, E. Damek, C. Yacoub,  Spherical analysis on harmonic $AN$ groups, Ann. Scuola
		Norm. Sup. Pisa \textbf{23} (1996), 643--679.
		
		\bibitem{AnShift} J.-Ph. Anker, P. Martinot, E. Pedon, A.G. Setti, The shifted wave equation
		on Damek–Ricci spaces and on homogeneous trees. M.A. Picardello. Trends in Harmonic Analysis
		(dedicated to A. Figà--Talamanca on the occasion of his retirement), 3, Springer--Verlag, pp.1--25, 2013, INdAM Ser.
		
		\bibitem{ACD97} F. Astengo, R. Camporesi, B. Di Blasio, The Helgason Fourier transform on a class of nonsymmetric harmonic spaces, \textit{Bull. Austral. Math. Soc.} \textbf{55} (1997), 405--424.
		
		\bibitem{AD08} F. Astengo, R. Camporesi, B. Di Blasio,
		Some properties of horocycles on Damek--Ricci spaces,
		\textit{Diff. Geom. Appl.} \textbf{26} (2008),  676--682.
		
		\bibitem{ADiB99} F. Astengo,  B. Di Blasio, A Paley--Wiener theorem on $NA$ harmonic spaces, \textit{Colloq. Math.} \textbf{80} (1999), 211--233.
		
		\bibitem{BS2022} T. Bruno, F. Santagati, Optimal heat kernel bounds and asymptotics on Damek--Ricci spaces. Preprint (2022). ArXiv: 2210.02066.
		
		
		\bibitem{Camp2014} R. Camporesi, The biradial Paley--Wiener theorem for the Helgason--Fourier transform on Damek--Ricci spaces, \textit{J. Funct. Anal.}, \textbf{267}, no.2 (2014), 428--451.
		
	\bibitem{Camp2016} R. Camporesi,
		The $\mathfrak{v}$--radial Paley--Wiener theorem for the Helgason Fourier transform on Damek--Ricci spaces, \textit{Colloq. Math.} \textbf{144} no.1 (2016), 87--113.
		
		\bibitem{CGK2017} J. Christensen, F. Gonzalez, T. Kakehi,
		Surjectivity of mean value operators on noncompact symmetric spaces,
		\textit{J. Funct. Anal.}, \textbf{272}, no.9 (2017), 3610--3646.
		
		\bibitem{GKCW2023} J. Christensen, F. Gonzalez, T. Kakehi, J. Wang, The snapshot problem for the wave equation, \textit{Adv. Math.}, \textbf{444} (2024).
		
		\bibitem{CDKR} M. G. Cowling, A. H. Dooley, A. Kor{\'a}nyi, F. Ricci, An approach to symmetric spaces of rank one via groups of Heisenberg type, \textit{J. Geom. Anal.} \textbf{8} (1998), 199--237.
		
		\bibitem{Dad79} J. Dadok, Paley--Wiener theorem for singular support of $K$--finite distributions on symmetric spaces, \textit{J. Funct. Anal.}, \textbf{31}, no.3 (1979), 341--354.
		
		\bibitem{D} E. Damek, Curvature of a semidirect extension of a Heisenberg type nilpotent group, \textit{Colloq. Math.} \textbf{53} (1987), 249--253.
		
		
		\bibitem{DR1} E. Damek, F. Ricci, A class of nonsymmetric harmonic Riemannian spaces, \textit{Bull. Amer. Math. Soc.} \textbf{27} (1992), 139--142.
		
		\bibitem{DR92} E. Damek, F. Ricci, Harmonic analysis on solvable extensions of $H$--type groups, \textit{J. Geom. Anal.} \textbf{2} (1992), 213--248.
		
		
		\bibitem{Eg73} M. Eguchi, M. Haszizume, K. Okamoto, The Paley--Wiener theorem for distributions on symmetric spaces, \textit{Hiroshima Math. J.} \textbf{3} (1973), 109--120.
		
		
		\bibitem{Ehr54} L. Ehrenpreis, Solution of some problems of division. I. Division by a polynomial of derivation, \textit{Amer. J. Math.} \textbf{76} (1954), 883--903.
		
		\bibitem{Ehr60} L. Ehrenpreis, Solution of some problems of division. IV. Invertible and elliptic operators, \textit{Amer. J.
			Math.} \textbf{82} (1960), 522--588.
		
		\bibitem{GWK2021} F. Gonzalez, J. Wang, T. Kakehi,
		Surjectivity of convolution operators on noncompact symmetric spaces, \textit{J. Funct. Anal.} \textbf{280}, no.2 (2021).
		
		\bibitem{GR2015} I.S. Gradshteyn, I.M. Ryzhik, \textit{Table of Integrals, Series, and Products}, eighth edition, Elsevier/Academic Press, Amsterdam, 2015, translated from the Russian, translation edited and with a preface by Daniel Zwillinger and Victor Moll, revised from the seventh edition.
		
		\bibitem{Hel64} S. Helgason, Fundamental solutions of invariant differential operators on symmetric spaces, \textit{Amer. J. Math.} \textbf{86} (1964), 565--601.
		
		\bibitem{Hel2000} S. Helgason, Geometric Analysis on Symmetric Spaces, second edition, Mathematical Surveys and Monographs, vol. 39, American Mathematical Society, Providence, RI, 2008.
		
		\bibitem{Horm1} L. H{\"o}rmander, The analysis of linear partial differential operators I, Second Edition, Springer--Verlag, 1990.
		
		\bibitem{Horm2} L. H{\"o}rmander, The analysis of linear partial differential operators  II, Springer--Verlag, 1983.
		
		
		\bibitem{Koo} T.H. Koornwinder, Jacobi Functions and Analysis on Noncompact Semisimple Lie Groups. In: Askey, R.A., Koornwinder, T.H., Schempp, W. (eds) Special Functions: Group Theoretical Aspects and Applications. Mathematics and Its Applications, vol 18. Springer, Dordrecht, 1984.
		
		\bibitem{Mal55} B. Malgrange, Existence et approximation des solutions des équations aux dérivées partielles
		et des équations de convolution, Ann. Inst. Fourier (Grenoble) \textbf{6} (1955–1956), 271--355.
		
		\bibitem{PS15} N. Peyerimhoff, E. Samiou, 
		Integral geometric properties of non-compact harmonic spaces. \textit{J. Geom. Anal.} \textbf{25} no.1 (2015),  122--148.
		
		
		\bibitem{RS09} S.K. Ray, R.P. Sarkar, Fourier and Radon transform on harmonic $NA$ Groups.
		\textit{Trans. Am. Math. Soc.} \textbf{361} no. 8 (2009), 4269--4297.
		
		\bibitem{Rou} F. Rouvière, Espaces de Damek--Ricci, géométrie et analyse, in: Analyse sur les
		groupes de Lie et théorie des représentations (Kénitra, 1999), Sémin. Congr. 7,
		pp. 45--100, Soc. Math. France, Paris, 2003.
		
	\end{thebibliography}
\end{document}